\documentclass{article}
\usepackage{geometry,graphicx} 
\usepackage{authblk}
\usepackage{subcaption} 
\usepackage{amsmath,amsthm,amssymb,commath,bm}
\usepackage{hyperref}
\usepackage{siunitx}
\usepackage{comment}

\usepackage[square,numbers]{natbib}
\hypersetup{pdfborder = 0 0 0, colorlinks=false,
 linkcolor=black,citecolor=black, filecolor=black, urlcolor=black, }
\usepackage{cleveref}
\usepackage[title]{appendix}
\newtheorem{lemma}{Lemma}
\newtheorem{theorem}{Theorem}
\newtheorem{definition}{Definition}
\newtheorem{remark}{Remark}
\newtheorem{corollary}{Corollary}
    
\usepackage{tabularx,array}
\newcolumntype{Y}{>{\raggedright\arraybackslash}X} 

\renewcommand{\v}[1]{\bm{#1}}
\newcommand{\vu}{{\v u}}
\newcommand{\vn}{{\v n}}
\newcommand{\vf}{{\v f}}
\newcommand{\vx}{{\v x}}

\newcommand{\vk}{{\v k}}
\newcommand{\vy}{{\v y}}

\newcommand{\eps}{\varepsilon}

\newcommand{\prol}{{\psi_0^c}}

\newcommand{\erf}{\operatorname{erf}}

\newcommand{\erfc}{\operatorname{erfc}}

\newcommand{\wfunc}{\phi} 
\newcommand{\wfunchat}{\widehat{\phi}} 
\newcommand{\wgauss}{\phi_G} 
\newcommand{\wgausshat}{\widehat{\phi}_G}
\newcommand{\wprol}{\phi_P} 
\newcommand{\wprolhat}{\widehat{\phi}_P}
\newcommand{\Kmax}{K_{\rm{max}}}

\newcommand{\truncrad}{\mathcal{R}}

\newcommand{\unitbox}{\mathcal{B}}
\newcommand{\Bres}{B_R}
\newcommand{\Breshat}{\widehat{B}_R}
\newcommand{\Bmoll}{B_M}

\newcommand{\Bhat}{\widehat{B}}

\newcommand{\Bmollhat}{\widehat{B}_M}
\newcommand{\Hres}{H_R}

\newcommand{\Hmollhat}{\widehat{H}_M}

\newcommand{\Kres}{{K_R}}
\newcommand{\Kmoll}{{K_M}}
\newcommand{\Kdiff}{{K_D}}
\newcommand{\Kmollhat}{\widehat{K}_M}
\newcommand{\Kdiffhat}{\widehat{K}_D}
\def\kmhat_#1{\widehat{K}_{M_{#1}}}
\def\km_#1{K_{M_{#1}}}
\def\kr_#1{K_{R_{#1}}}
\def\kd_#1{K_{D_{#1}}}

\def\bmohat_#1{\widehat{B}_{M_{#1}}}
\def\brhat_#1{\widehat{B}_{R_{#1}}}
\def\bdhat_#1{\widehat{B}_{D_{#1}}}
\def\bmo_#1{B_{M_{#1}}}
\def\br_#1{B_{R_{#1}}}
\def\bd_#1{B_{D_{#1}}}

\newcommand{\vm}{\mathbf{m}}

\newcommand{\R}{\mathbb{R}}

\newcommand{\tind}{\beta}
\newcommand{\sind}{\alpha}
\newcommand{\xb}{{\v x}_{\beta}}

\newcommand{\ulocal}{u_{\text{local}}}
\newcommand{\ufar}{u_{\text{far}}}
\newcommand{\uself}{u_{\text{self}}}

\newcommand{\RSdiag}{S_{\text{diag}}}
\newcommand{\RSoffd}{S_{\text{offd}}}
\newcommand{\RTdiag}{T_{\text{diag}}}
\newcommand{\RToffd}{T_{\text{offd}}}
\newcommand{\ROoffd}{\Omega_{\text{offd}}}

\newcommand{\MSdiag}{S_{M,\text{diag}}}
\newcommand{\MSoffd}{S_{M,\text{offd}}}
\newcommand{\MTdiag}{T_{M,\text{diag}}}
\newcommand{\MToffd}{T_{M,\text{offd}}}

\newcommand{\FSdiag}{S_{F,\text{diag}}}
\newcommand{\FSoffd}{S_{F,\text{offd}}}
\newcommand{\FTdiag}{T_{F,\text{diag}}}
\newcommand{\FToffd}{T_{F,\text{offd}}}

\newcommand{\wfun}{\gamma} 
\newcommand{\wfunhat}{\widehat{\gamma}} 
\newcommand{\bhkhat}{\widehat{B}}

\newcommand{\vzero}{{\v 0}}

\newcommand{\be}{\begin{equation}}
\newcommand{\ee}{\end{equation}}
\newcommand{\ba}{\begin{aligned}} 
\newcommand{\ea}{\end{aligned}}

\newcommand{\MDU}{Division of Mathematics and Physics, M\"alardalen University, V\"aster{\aa}s, Sweden}
\newcommand{\CCM}{Center for Computational Mathematics, Flatiron Institute, Simons Foundation, New York, NY, 10010, USA}
\newcommand{\COURANT}{Courant Institute of Mathematical Sciences, New York University, New York, New York 10012, USA}
\newcommand{\KTH}{KTH Mathematics, Royal Institute of Technology, SE-100 44 Stockholm, Sweden}

\title{Fast summation of Stokes potentials using a new kernel-splitting in the DMK framework}
\author[1]{Ludvig af Klinteberg}
\author[2,3]{Leslie Greengard}
\author[2]{Shidong Jiang}
\author[4]{Anna-Karin Tornberg}

\affil[1]{\MDU}
\affil[2]{\CCM}
\affil[3]{\COURANT}
\affil[4]{\KTH}

\begin{document}

\maketitle

\begin{abstract}
Classical Ewald methods for Coulomb and Stokes interactions rely on ``kernel-splitting," using 
decompositions based on Gaussians to divide the resulting
potential into a near field and a far field component.
Here, we show that a more efficient splitting for the scalar biharmonic Green's function can be derived using zeroth-order prolate spheroidal wave functions (PSWFs), which in turn yields new efficient splittings for the Stokeslet, stresslet, and elastic kernels, since these Green's tensors can all be derived from the biharmonic kernel. 
This benefits all fast summation methods based on kernel splitting, including FFT-based Ewald summation methods, that are suitable for uniform point distributions, and DMK-based methods that
allow for nonuniform point distributions.
The DMK (dual-space multilevel kernel-splitting)
algorithm we develop here is
fast, adaptive, and linear-scaling,
both in free space and in a periodic cube. We demonstrate its performance with numerical examples in two and three dimensions.
\end{abstract}

\section{Introduction}\label{sec:intro}
The incompressible Stokes equations
\be
\ba
-\Delta \vu + \nabla p &= \vf, \\
\nabla \cdot \vu &= 0
\ea
\ee
model viscous flow when inertial forces can be neglected
in the Navier–Stokes equations. Such ``Stokes flow" arises 
in many contexts, including the
sedimentation of small particles in a viscous fluid, flow through narrow capillaries
or microchannels, lubrication problems, the locomotion of microorganisms, etc.
The fundamental solution -- the Stokeslet -- is 
defined by the formula
\be
S_{jl}(\vx) = (-\delta_{jl}\nabla^2 + \nabla_j\nabla_l) B(\vx)
=\left\{\begin{aligned}& 
\frac{1}{4\pi}\left(-\delta_{jl} \log r + \frac{x_jx_l}{r^2}\right),  & \vx \in \mathbb{R}^2, \\
&\frac{1}{8\pi} \left(\frac{\delta_{jl}}{r} + \frac{x_jx_l}{r^3}\right), & \vx \in \mathbb{R}^3,
\end{aligned}\right.
\label{eq:stokeslet_def_deriv}
\ee
where $\delta_{ij}$ is the Kronecker delta, $r=|\vx|$, $B$ is the biharmonic
Green's function, i.e., $\Delta^2 B(\vx) = \delta(\vx)$, and\footnote{The factor $-3/2$ in the bracket in the $\mathbb{R}^2$-expression is not determined by $\Delta^2 B(\vx) = \delta(\vx)$, but is set so that the differentiations in \eqref{eq:stokeslet_def_deriv} and \eqref{eq:stresslet_def_deriv} hold.}
\be
B(\vx) = \left\{
\begin{aligned}& 
\frac{1}{8\pi}\left(r^2 \log r -\frac{3}{2}r^2\right),  & \vx \in \mathbb{R}^2, \\
&-\frac{r}{8\pi}, & \vx \in \mathbb{R}^3.
\end{aligned}\right.
\label{eq:biharmonic_green_def}
\ee
The associated stresslet $T$ is defined by the formula
\be
T_{jlm}(\vx) =
\left[ -\left(
      \delta_{jl}\nabla_m+\delta_{lm}\nabla_j+\delta_{mj}\nabla_l
    \right) \nabla^2 + 2\nabla_j\nabla_l\nabla_m \right] B(\vx)
=\left\{ \begin{aligned} &-\frac{1}{\pi} \frac{x_jx_lx_m}{r^4}, & \vx \in \mathbb{R}^2, \\
&-\frac{3}{4\pi} \frac{x_jx_lx_m}{r^5}, & \vx \in \mathbb{R}^3.\end{aligned}\right.
\label{eq:stresslet_def_deriv}
\ee
The Stokeslet and stresslet are widely used in 
numerical methods for boundary value problems governed 
by the Stokes equations using either integral equation methods or 
the method of fundamental solutions, and for simulating low Reynolds hydrodynamic interactions in particle suspensions (see, for example, \cite{AfKlinteberg2014a,Brady1988,Broms2025,Jung2006,Lefebvre-Lepot2015,Palsson2020,sorgentone_highly_2018,Wang2016,Wu2020}).

In this paper, we consider the computation of free-space point sums
of the form
\be
u_j(\vx_\tind) = \sum_{\substack{\sind=1\\ \sind \neq \tind}}^N \sum_{l=1}^d 
S_{jl}(\vx_\tind-\vx_\sind) 
{\bf f}_{\sind,l} \quad {\rm or} \quad
u_j(\vx_\tind) = \sum_{\substack{\sind=1\\ \sind \neq \tind}}^N \sum_{l=1}^d \sum_{m=1}^d
T_{jlm}(\vx_\tind-\vx_\sind) 
{\bf f}_{\sind,l} {\bf n}_{\sind,m} 
\label{eq:stokes_stress_sum}
\ee
for $\tind=1,\dots,N$, where ${\bf f}_{\sind,l}$ are the components
of  a force vector and ${\bf n}_{\sind,l}$ are the components of a
prescribed orientation vector (typically the normal to some surface). 
For the sake of simplicity, we assume that the sources and targets lie in the unit box
$\unitbox$ centered at the origin: $\unitbox = [-1/2,1/2]^d$.
Eq.~\eqref{eq:stokes_stress_sum} is a special case of the more general discrete convolution with an arbitrary interaction kernel $K$:
\be
u(\vx_\tind) = \sum_{\substack{\sind=1\\ \sind \neq \tind}}^N K(\vx_\tind-\vx_\sind) \rho_\sind, \quad \tind=1,\dots,N.
\label{eq:fs_sum}
\ee
We will refer to $u$ as a (vector-valued) potential in $\mathbb R^d$.
The source strength $\rho_\sind$
is a vector in $\mathbb R^d$
for the Stokeslet (a rank-two tensor) and a matrix in $\mathbb R^{d \times d}$ 
for the 
stresslet (a rank-three tensor).
We will also consider the periodic version of \eqref{eq:fs_sum}:
\begin{align}
  u(\vx_\tind) = \sum_{\sind=1}^N \sum_{\v p\in\mathbb Z^d}^* K(\vx_\tind-\vx_\sind + \v p) 
  \rho_\sind, \quad \tind=1,\dots,N,
  \label{eq:periodic_sum}
\end{align}
where the $*$ superscript indicates that the term $\v p=0$ is ignored
when $\vx_\tind=\vx_\sind$.

In general, the algorithms presented here will make use of a decomposition of the 
interaction kernel $K$:
\be
K(\v x)=\Kmoll(\v x)+\Kres(\v x),
\label{eq:singlelevel_ks}
\ee
where $\Kmoll(\v x)$ is a {\em mollified} kernel 
which captures the far field correctly and decays rapidly in Fourier space,
and $\Kres(\v x)$ is a localized 
{\em residual} kernel that corresponds to
a near field correction. For periodic electrostatic interaction, 
Ewald \cite{Ewald1921} in 1921 introduced the splitting 
for the three-dimensional Coulomb kernel $H(|\vx|)$:
\begin{align}
    H(r) = \frac{1}{4 \pi r} = \frac{\erf(r/\sigma)}{4 \pi r} + 
\frac{\erfc(r/\sigma)}{4 \pi r}
    := H_M(r) + H_R(r). 
    \label{eq:ewald-split}
\end{align}
where erf and erfc are the error and complementary error functions
\[ \erf(x) = \frac{2}{\sqrt{\pi}} \int_0^x e^{-t^2} \, dt , \qquad 
\erfc(x) = 1 - \erf(x),
\]
and $\sigma$ is a free parameter.
The Fourier transform of $H_M(r)$ is well-known:
\[ 
\widehat{H}_M(k) = \frac{e^{-k^2\sigma^2/4}}{k^2},
\]
so that $\widehat{H}_M(k)$ decays rapidly in Fourier space.
$H_R(r)$, on the other hand, maintains the 
singularity of the Coulomb kernel at the origin but can be truncated in physical space
due to its rapid decay. 

Ewald summation is typically accelerated with the fast Fourier transform (FFT) to compute smooth, long-range interactions with the mollified kernel $\Kmoll$, as described in detail for Stokes kernels
in~\cite{AfKlinteberg2016fse,AfKlinteberg2014a,bagge_fast_2023,Lindbo2010,Palsson2020,Saintillan2005}.  
More precisely, the long-range contribution is evaluated by spreading the sources onto a uniform grid, applying an FFT, then performing an appropriate scaling followed by an inverse FFT, after which the result is interpolated from the uniform grid to the evaluation points.
For readers who are familiar with the nonuniform FFT
(NUFFT)~\cite{Barnett2019finufft,finufftlib,Dutt1993,Dutt1995,potts2004}, one can note that this procedure can be seen as a type-1 transform, followed first by a scaling in Fourier space and then a type-2 transform.
In contrast to the NUFFT, oversampling is not needed in either of the transforms in Ewald summation, due to the mollification in Fourier space
\cite{Darden1993,liang2025arxiv}, although a small amount may be used for improved overall efficiency \cite{nestlerParameterTuningNFFT2016}.

With $O(N)$ grid points and Fourier modes, this is an $O(N \log N)$ procedure. The residual term, confined to nearby neighbors, can usually be summed directly in $O(N)$ time when the sources are more or less uniformly distributed. If the data are highly clustered, however, it is very difficult to balance the long-range mollified part and the local residual part effectively. In this case, the overall cost may be pushed toward $O(N^2)$.

To overcome this, tree-based algorithms such as
the fast multipole method 
(FMM)~\cite{greengard1988,Greengard1987,Tornberg2008,Ying2004} 
and multilevel summation 
\cite{brandt1990jcp,brandt1998,multilevel_summation_2015,tensor_multilevel_ewald} 
were developed over the last few decades.
Both methods are straightforward to apply on adaptive data structures.
In \cite{Ying2004}, a kernel-independent FMM was developed for kernels that
are fundamental solutions of some second-order elliptic partial differential equation.
In \cite{Tornberg2008}, an FMM was developed for the Stokeslet in three
dimensions using four (scalar) FMM calls for the Laplace kernel. Excluding the sorting step,
which usually accounts for a small fraction of the total computational time, FMMs typically require
$O(N)$ work. They rely on the use of {\em translation operators} to compute
well-separated interactions at each level in a tree hierarchy.
More recently \cite{jiang2025cpam},
the dual-space multilevel kernel-splitting (DMK) framework was developed,
unifying multilevel summation and the FMM with the Fourier-based
convolution that underlies Ewald summation. 
DMK relies on a multilevel extension of Ewald splitting:
\be
K(\vx)= \km_0(\vx)
 +\sum_{\ell=0}^{L-1} \kd_\ell(\vx) +\kr_L(\vx)  \qquad (L \geq 1)
\label{eq:multilevel_ks}
\ee
where
\be
\kd_\ell(\v x) =\km_{\ell+1}(\v x) -\km_\ell(\v x).
\label{eq:diffkernel}
\ee
While we postpone a more detailed description of the full numerical method to \cref{sec:fast},
we associate the terms in the telescoping approximation with the refinement level in the tree
hierarchy. In particular,
level $0$ is defined to be the unit cube itself, and level $\ell$ is obtained by
subdividing each box at level $\ell-1$ into $2^d$ child boxes of side length 
$r_{\ell} =2^{-\ell}$.
We will refer to $r_\ell$ as the {\em box length}.
$\km_\ell(\vx)$ in 
\eqref{eq:multilevel_ks}, \eqref{eq:diffkernel} will be the mollified kernel tied to the length 
scale $r_{\ell}$. For scale-invariant kernels, the function $\km_\ell(\vx)$ in \eqref{eq:multilevel_ks} is obtained from $\Kmoll(\vx)$ in \eqref{eq:singlelevel_ks} by scaling both its arguments and its values,
and the {\em difference} kernel $\kd_\ell(\vx)$ is obtained via \eqref{eq:diffkernel}.

By careful construction, as we shall see below, 
the {\em difference} kernel $\kd_\ell(\v x)$ 
can be designed so that it is localized to nearest neighbor
boxes at level $\ell$ in the tree hierarchy.
For readers familiar with the FMM, it is worth noting that
the data flow and bookkeeping tasks are simpler in DMK.
The ``interaction list" is simply the $3^d$ near neighbors at each level rather than the 
$6^d-3^d$  well-separated boxes that permit the use of multipole expansions.
Unlike the FMM, DMK does not make direct use of the governing
PDE. It relies only on separation of variables with careful control of the Fourier expansion length
at each level. For non-oscillatory kernels, these are generally 
of the order $O(\log^d(1/\eps))$.

Since the Stokeslet and stresslet can be derived from the biharmonic Green's function,
we only need to consider the splitting for the biharmonic kernel. In \cite{Hasimoto1959},
Hasimoto proposed a split for the Stokeslet that corresponds to the following split for the three-dimensional biharmonic
kernel
\be
B(r) = \Bmoll(r) + \Bres(r),
\ee
with the residual kernel $\Bres$ expressed in physical space and the mollified
kernel $\Bmoll$ expressed in the Fourier domain: 
\begin{align}
\begin{split}
  \Bres(r) &= -\frac{r}{8\pi} \erfc(r / \sigma) + \frac{\sigma}{8 \pi^{3/2}} e^{-r^2/\sigma^2},
  \\
  \Bmollhat(k) &= \frac{1}{k^4} e^{-k^2 \sigma^2/4} \del{1 + \frac{k^2 \sigma^2}{4} }.
\end{split}
\label{eq:hasimoto-split}
\end{align}
See \cite{AfKlinteberg2016fse,bagge_fast_2023} and the references therein for 
a detailed discussion.

It was shown in \cite{jiang2025cpam} that telescoping kernel approximations
can be developed in either physical or Fourier space. For the Stokeslet and the stresslet,
it will be convenient to develop dimension-independent kernel decomposition in Fourier space.
It was also shown \cite{jiang2025cpam} that replacing Gaussians 
with prolate spheroidal wave functions (PSWFs) leads to more efficient algorithms 
for a variety of scalar interaction kernels including the Coulomb kernel.
The basic reason is that, {\em for the same spatial support},
 the zeroth-order prolate spheroidal wave function
(PSWF)~\cite{slepian1961bstj,landau1961bstj,slepian1978bstj,slepian1983sirev},
denoted by $\psi_0^c$, has a Fourier transform with a bandlimit about half that of 
a Gaussian  (see \cref{sec:prol}) at high precision.

\begin{remark}
The rotlet, the fundamental solution of the Stokes equations corresponding to
a point torque, is given by the formula
\be
\Omega_{jl}(\vx)
=\left\{\begin{aligned}
& -\frac{1}{2} \epsilon_{jl}\nabla_l H(r) = \frac{1}{4\pi} \frac{\epsilon_{jl}x_l}{r^2}, & \vx \in \mathbb{R}^2,\\
& -\frac{1}{2} \epsilon_{jlm}\nabla_m H(r) = \frac{1}{8\pi} \epsilon_{jlm} \frac{x_m}{r^3}, & \vx \in \mathbb{R}^3,
\end{aligned}\right.
\label{eq:rotlet_def_deriv}
\ee
where $H$ is the harmonic Green's function
\be
H(\vx) = \left\{
\begin{aligned} 
&-\frac{1}{2\pi}\log r, & \vx \in \mathbb{R}^2, \\
&\frac{1}{4\pi r}, & \vx \in \mathbb{R}^3.
\end{aligned}\right.
\label{eq:harmonic_green_def}
\ee
As it relies only on the harmonic kernel, it can be treated using the original DMK method
of \cite{jiang2025cpam}.
\end{remark}

\begin{remark}
The elastic kernel, the fundamental solution for the Navier-Cauchy equations
in isotropic linear elasticity:
\be
\mu \Delta \vu+(\lambda+\mu)\nabla(\nabla\cdot\vu)+\vf \delta(\vx)=0,
\ee
is given by the formula
\be
\Gamma_{jl}(\vx) = \left\{\begin{aligned}
&-\frac{1}{4\pi \mu} \left( \alpha \delta_{jl} \log r  + \frac{x_j x_l}{r^2} \right),
& \vx \in \mathbb{R}^2,\\
&\frac{1}{16\pi \mu (1 - \nu)} \left( \frac{\delta_{jl}}{r} + \alpha \frac{x_j x_l}{r^3} \right), & \vx \in \mathbb{R}^3.
\end{aligned}\right.
\ee
Here, $\lambda$ and $\mu$ are Lam{\' e} parameters, $\nu=\frac{\lambda}{2(\lambda+\mu)}$
is Poisson's ratio, $\alpha = \frac{1 - 2\nu}{1 - \nu}$. Since the 
elastic kernel is essentially the same as the Stokeslet (up to constant factors), it can be treated
using the methods developed here with trivial modification.
Finally, we should note that many boundary value problems in linear elasticity
can be formulated directly in terms of a scalar biharmonic potential
(see, for example, \cite{jiang2011jcp}). The rapid evaluation of layer potentials
with kernels which are derivatives of $B(r)$ can also be accomplished using the methods
developed below.
\end{remark}

The main contributions of the present paper are as follows.
\begin{itemize}
\item Since the biharmonic kernel is a PDE kernel, we start from the Fourier transform
  of the biharmonic kernel:
\begin{align}
\widehat B(k) = \frac{1}{k^4}, \qquad k = |\vk|.
\label{eq:biharmonic_ft}
\end{align}
We develop a rigorous analysis for fairly general splittings of the biharmonic kernel.
The analysis provides sufficient conditions for the residual kernel to be
compactly supported in physical space.

\item Based on this analysis, we propose a 
new splitting for the biharmonic kernel in the Fourier domain of the form:
\be
\ba
\Bhat(k) &= \Bmollhat(k) + \Breshat(k),\\
\Bmollhat(k) &= \widehat B(k) \wfunhat(k)
,\\
\Breshat(k) &= \widehat B(k) (1- \wfunhat(k))
\ea
\label{eq:biharmonicks}
\ee
where
\be
\begin{split}
\wfunhat(k) = \wfunchat(k)-k\wfunchat'(k)/2, \qquad \wfunchat(k) = \frac{\psi_0^c(k/c)}{\psi_0^c(0)} ,
\label{eq:bsplitdef}
\end{split}
\ee
and $\psi_0^c$ is the prolate spheroidal wave function of order zero.
This leads directly to more efficient
kernel splittings for the Stokeslet and stresslet \eqref{eq:stokeslet_def_deriv}
and \eqref{eq:stresslet_def_deriv} via differentiation, compared with the existing
Hasimoto splitting~\eqref{eq:hasimoto-split} based on Gaussians.
\item We use the new kernel splitting to accelerate 
  the DMK algorithm in the adaptive setting. (In the uniform setting,
  the single-level splitting $B(\v x)=\Bmoll(\v x)+\Bres(\v x)$
  leads to an improved version of
  fast Ewald summation.)
\item We describe the modifications to the DMK framework
  required to impose periodic boundary conditions on a cubic domain using a level-restricted adaptive tree.
\end{itemize}

The mathematical background of our approach 
is described in \cref{sec:prelim}.
\Cref{sec:biharmonicks} is devoted to a general theory for choosing the
{\em biharmonic mollifier}
$\wfunhat(k)$, above, in the Fourier domain.
We then show, in \cref{sec:prolks}, 
how prolate spheroidal
wave functions (PSWFs) lead to the
highly efficient choice \eqref{eq:bsplitdef}.
In \cref{sec:fast}, we present fast algorithms for the Stokes and stresslet,
especially the DMK algorithm and its modifications for periodic boundary conditions.
Numerical results are presented in 
\cref{sec:results} with concluding remarks in \cref{sec:conclusions}.

\section{Mathematical Preliminaries} \label{sec:prelim}

\subsection{Fourier transforms}
For $\vx, \vk \in \mathbb R^d$, the
Fourier transform pair $f(\vx), \widehat{f}(\vk)$ (under suitable conditions)
is defined by
\begin{align}
  \widehat f(\v k) &= \int_{\mathbb R^d} f(\v x) e^{-i \vk\cdot\vx} \dif \v x, \\
  f(\v x) &= \frac{1}{(2\pi)^d} \int_{\mathbb R^d} \widehat f(\vk) e^{i \vk\cdot\vx} \dif \vk. 
\end{align}
When the context is clear, we will often use the notation $r=|\v x|$ and $k=|\vk|$. For a
radially symmetric function in $\mathbb R^d$ we write
$f(\v x) = f(r)$ and $\widehat f(\vk) = \widehat f(k)$, with Fourier
transform pair
\be
\ba
  \widehat f(k) = 2\pi \int_0^\infty J_0(kr) f(r) r \dif r \quad &\leftrightarrow \quad
  f(r) = \frac{1}{2\pi} \int_0^\infty J_0(kr) \widehat f(k) k \dif k 
\quad {\rm in} \ \ \mathbb R^2 ,\\
  \widehat f(k) = 4\pi \int_0^\infty \frac{\sin(kr)}{kr} f(r) r^2 \dif r \quad &\leftrightarrow
 \quad f(r) = \frac{1}{2\pi^2} \int_0^\infty \frac{\sin(kr)}{kr} \widehat f(k) k^2 \dif k 
\quad {\rm in} \ \ \mathbb R^3,
\ea
\label{eq:radial_ft}
\ee
where $J_0$ is the Bessel function of order zero.

\subsection{Quadrature in the Fourier domain} \label{sec:trapquad}

Let $u(\vx) = \rho(\vx) \ast \Kdiff(\vx)$, where $\rho(\vx)$ is defined as the 
point source distribution 
\[ \rho(\vx) = \sum_{\sind=1}^N \delta(\vx-\vx_\sind) \rho_\sind \, ,
\]
and $\delta(\vx)$ is the usual Dirac $\delta$-function.
From the convolution theorem, we may write
\be
  u(\v x) = \frac{1}{(2\pi)^d} \int_{\mathbb{R}^d} 
\widehat u(\vk) e^{i \vk\cdot\vx} \dif \vk,
\ee
where $\widehat{u}(\vk) = \widehat{\rho}(\vk) \, \Kdiffhat(\vk)$ with
\be
 \widehat{\rho}(\vk_{\vm}) = 
\sum_{\sind=1}^N \rho_\sind e^{i \vk_{\vm} \cdot \vx_\sind} .
\ee
Suppose now that $\Kmax$ has been chosen so that
\be
  u(\v x) = \frac{1}{(2\pi)^d} \int_{[-\Kmax,\Kmax]^d} 
\widehat u(\vk) e^{i \vk\cdot\vx} \dif \vk
+ O(\epsilon)
\label{truncft}
\ee
with $\|\vx\|_\infty\leq 1/2$.
That is, $u(\vx)$ is bandlimited to precision $\epsilon$ and restricted to the unit cube.
Since $\widehat{\rho}(\vk)$ is not decaying as $|\vk| \rightarrow \infty$ the bandlimit
is determined by the behavior of $\Kdiffhat$. Given $\Kmax$ we would like to know
how many quadrature points are needed in order to compute
\eqref{truncft} to precision $\epsilon$. Assuming $\Kdiffhat$ itself is nonoscillatory,
it is the range of $\vx$ that determines Nyquist sampling of the integrand. 
More precisely, the number of oscillations of the integrand over the range 
$[-\Kmax,\Kmax]$ is $O(x_{max} \cdot \Kmax)$ where $x_{max} = \|\vx \|_\infty$.
Then, by choosing $n \approx x_{max} \cdot \Kmax$, 
with $\vm = (m_1,...,m_d)$,
$h = \frac{\Kmax}{n}$, and $\vk_{\vm} = h \vm$, it is well-known 
\cite{Trefethen2014} that as $\Kdiffhat$ is smooth the trapezoidal approximation
\be
  u(\v x_\tind) \approx \left( \frac{h}{2\pi}  \right)^d \sum_{\vm \in [-n,n]^d} 
\widehat{\rho}(\vk) \, \Kdiffhat(\vk)
e^{i \vk_{\vm}\cdot\vx_\tind}, \quad \tind = 1,\dots,N, 
\label{ewalds3}
\ee
is spectrally accurate,
with an error that decays superalgebraically as $n$ increases beyond 
$x_{max} \cdot \Kmax$.

\begin{remark} \label{rmk:scaling}
Below, we will construct Fourier signatures $\Kdiffhat(\vk)$ with a 
bandlimit of the order $O(1)$ to precision $\epsilon$ in both physical and Fourier space.
We will use those signature functions in a multilevel fashion such that, at level $\ell$
we rescale the Fourier signature to be $\Kdiffhat(\vk r_\ell)$ while 
we simultaneously reduce the spatial scale to be $O(r_\ell)$. In this manner, the product
$x_{max} \cdot \Kmax$ will remain {\em constant} and the number of points needed 
can be shown to be of the order $O( \log^d(1/\epsilon))$. 
See sections \ref{sec:biharmonicks} and \ref{sec:par_selection} below, and
\cite{jiang2025cpam} for a more detailed discussion.
\end{remark}

\subsection{Truncated biharmonic kernels}
One of the difficulties with the kernels of classical physics, including the biharmonic
kernel, is that they have singularities in the Fourier domain. In the biharmonic case,
the singularity $(1/k^4)$ lies at the origin.
Since we are assuming the sources and targets lie within the unit cube $\unitbox$,
it is clear that multiplying the kernel by the function
${\rm rect} \left( \frac{r}{\truncrad} \right)$ has no effect on interactions within $\unitbox$, where
$\truncrad = (1+\sqrt{d})$ and
\begin{align}
  \operatorname{rect}(x) =
  \begin{cases}
    1 & \text{for } |x| \le 1 \\
    0 & \text{for } |x| > 1.
  \end{cases}
\end{align}
In \cite{Vico2016}, it was shown that the Fourier transforms of these ``truncated"
kernels can be computed in closed form and are infinitely 
differentiable, so that the trapezoidal rule can be applied in the Fourier domain with 
high order accuracy as discussed in \cref{sec:trapquad}.
The truncated biharmonic
kernels derived in \cite{Vico2016}, however, decay at the rates $O(k^{-1})$ and $O(k^{-2})$ in 2D/3D rather than at the rate
$O(k^{-4})$ of the original biharmonic kernel, making the error in truncating the Fourier
transform suboptimal.
In \cite[Appendix~D]{bagge_fast_2023}, exploiting the non-uniqueness of the Green's function,
the authors proposed alternate biharmonic kernels (with truncation) of the form
\be
B^\truncrad(r) = \left\{\begin{aligned} &\frac{1}{8\pi}
\left(r^2\log r -\frac{r^2}{2} -r^2\log\truncrad +\frac{\truncrad^2}{2}\right)
\operatorname{rect}\del{\frac{r}{\truncrad}}, \qquad d=2,\\
&-\frac{1}{8\pi}\del{r - \frac{\truncrad}{2} -\frac{1}{2\truncrad} r^2}
\operatorname{rect}\del{\frac{r}{\truncrad}}, \qquad d=3.
\end{aligned}\right.
\label{eq:Btrunc}
\ee
The corresponding Fourier transform $\widehat B^\truncrad$ is given by
\be
\widehat B^\truncrad(k) =\left\{\begin{aligned}&\frac{1}{k^4}
\left(1 -J_0(\truncrad k)-\frac{1}{2}\truncrad k J_1(\truncrad k)\right), \qquad d=2, \\
&\frac{1}{k^4}\del{1 + \frac{1}{2}\cos(k\truncrad) - \frac{3}{2}\frac{\sin(k\truncrad)}{k\truncrad}}, \qquad d=3,
\end{aligned}\right.
\label{eq:Btrunc_hat}
\ee
where $J_n$ is the Bessel function of order $n$.  Note that $\widehat
B^\truncrad$ decays at the rates $O(k^{-3})$ and $O(k^{-4})$.  Note also that some care
must be taken in using the differential relation for the Stokeslet in
\eqref{eq:stokeslet_def_deriv}, since using $B^\truncrad(r)$ in place
of $B(r)$ results in an additional constant term which must be
corrected for \cite{bagge_fast_2023}, see \cref{sec:totsum}.

\subsection{Prolate spheroidal wave functions} \label{sec:prol}

The prolate spheroidal wave functions (PSWFs) of the first kind
are eigenfunctions of the Fourier integral operator
\be
\mathcal{F}[\sigma](x) = \int_{-1}^1 \sigma(t) e^{ictx} \dif t
\ee
for a parameter $c>0$. The PSWF of order $0$, denoted $\psi_0^c$,
is the eigenfunction 
corresponding to the largest eigenvalue $\lambda_0^c>0$.
It is well-known that  $\psi_0^c$ is an even function. PSWFs (or {\em prolates})
were studied extensively in a seminal sequence of papers
\cite{slepian1961bstj,landau1961bstj,slepian1978bstj,slepian1983sirev}.
They are important in the
context of kernel-splitting due to the following optimality property:
for all bandlimited functions 
with band limit $c$ 
defined by the formula 
\begin{align}
    f(x)=\int_{-1}^1 \sigma(t) e^{icxt} \dif t
\end{align}
for some $\sigma \in L^2[-1,1]$ with
\begin{align}
    \int_{-1}^1 |\sigma(t)|^2 \dif t=1,
\end{align}
we have~\cite[Thm.~3.53]{osipov_prolate_2013} 
\begin{align}
    \norm{f}_{L^2[-1,1]} = \int_{-1}^1 |f(x)|^2 \dif x \le |\lambda_0^c|^2,
    \label{eq:L2sq_f}
\end{align}
and \eqref{eq:L2sq_f} holds with equality only if $\sigma=\prol$, i.e.,
\be
\lambda_0^c \prol(x) = \int_{-1}^1 \prol(t) e^{ictx} \dif t.
\ee
Thus,
\begin{align}
  \widehat\psi_0^c(k) 
       = \lambda_0^c \prol(k/c), \quad |k| \le c .
         \label{eq:PSWFhat}
\end{align}
Setting $x=0$ on both sides of the above equation, we obtain
\be
\lambda_0^c \prol(0) = \int_{-1}^1\prol(t) \dif t.
\ee
Let us now compare the optimized window functions $\wprolhat^c$ and the Gaussian
$\wgausshat$:
\begin{align}
  \wprolhat^c(k) &= \frac{\prol(k/c)}{\prol(0)},
  &
    \wgausshat(k) &= e^{-k^2 \sigma^2/4},
 \label{eq:wfunchat_prol_G}
  \\
  \wprol^c(r) &= \frac{\prol(r)}{\int_{-1}^1 \prol(t)\dif t} ,
  &
    \wgauss(r) &= \frac{1}{\sigma\sqrt{\pi}}e^{-r^2/\sigma^2}, 
    &
    \label{eq:wfunc_prol_G}
\end{align}
assuming that the window functions are truncated at $r=1$ in real space
with an error defined to be the absolute value of the window at the point of truncation.
This error is controlled by the parameters $c$ and $\sigma$ which, in turn,
determine how many Fourier modes are required. For 10 digits of accuracy, $c=32$ is sufficient
in the case of the prolate window and $\sigma = 0.2$ is sufficient in the case of the 
Gaussian. Once $c$ and $\sigma$ are known, it remains to determine how many Fourier
modes are needed in each case.  For $\wprolhat^c$, the Fourier integral
is truncated (by definition) at $K_P=c$. For the Gaussian, the desired accuracy is achieved
at $K_G= \frac{2}{\sigma^2} \approx 50$.
That is, the number of modes needed to discretize the $d$-dimensional Fourier transform of the 
prolate window in each dimension is roughly $60 \%$ of that required for the
Gaussian.  For $d=3$, this results in a 4-fold reduction in the number of Fourier modes needed. 
\Cref{fig:split_compare}, modified from \cite[fig.~7]{jiang2025cpam}, compares
the stresslet splitting kernel \eqref{eq:T-offd} when based on either the prolate window or 
the Gaussian.
The number of Fourier modes needed is reduced by a factor $0.63^d$. 
\begin{figure}[ht]
  \centering
  \includegraphics[width=0.31\textwidth]{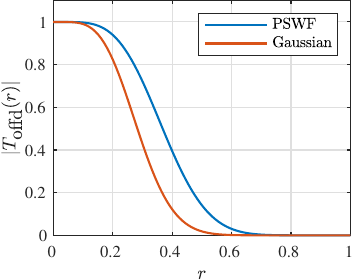}
  \hfill
  \includegraphics[width=0.31\textwidth]{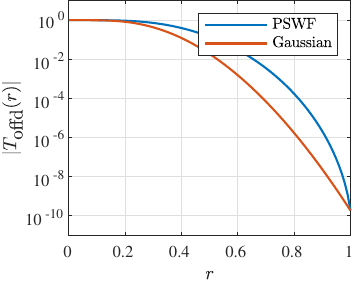}
  \hfill
  \includegraphics[width=0.31\textwidth]{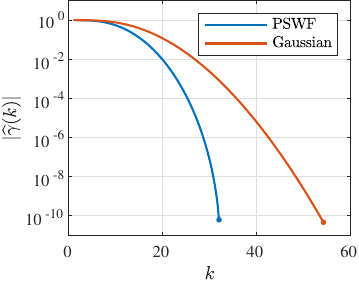}
  \caption{Efficiency of stresslet splitting using either the prolate window 
    ($\wprol$) or the Gaussian ($\wgauss$). The magnitude of the dominant term
    $\RToffd(r)$ in \eqref{eq:T-offd}
    is shown in the left and center plots in real space on a linear and $\log$ scale,
    respectively. The corresponding mollifier $\widehat\gamma(k)$ is plotted
    in the Fourier domain (see \cref{def:wfunhat}) on the right.
    For the same decay at $r=1$, we require 
    $c=32$ for $\wprol$) and $\sigma=0.192$ for $\wgauss$).
    The parameter choices and error analysis are described in more detail 
    in \cref{sec:par_selection} for all the relevant tensor kernels at
    various accuracy levels.}
  \label{fig:split_compare}
\end{figure}

\section{General kernel splitting for the biharmonic kernel}
\label{sec:biharmonicks}

We now develop a rather general (multilevel) splitting for the biharmonic kernel.
As observed in \cref{sec:intro}, we will only need to consider
the single-level splitting for the Fourier transform of the biharmonic
kernel~\eqref{eq:biharmonic_ft} of the form:
\be
\bhkhat(k) = \Bmollhat(k) + \Breshat(k).
\label{eq:bh_kernelsplit}
\ee

\begin{definition} \label{def:wfunhat}
  A function $\wfun(r)$ is called a {\em biharmonic mollifier} if it is
  an even, radially symmetric, 
  and its Fourier transform 
  $\wfunhat$ satisfies the following two properties:
  \be
  \ba
  \wfunhat(0) &= 1,\\
  \wfunhat''(0) & = 0.
  \ea\label{eq:wfunproperties}
  \ee
\end{definition}

\begin{lemma}
  Suppose that $\wfun$ is a biharmonic mollifier. Then $\wfunhat$ is 
  even and radially symmetric.
  Furthermore, $\wfun$ itself satisfies the following two properties:
  \begin{align}
    \int_{\mathbb{R}^d} \wfun(|\vx|) d\vx &= 1, \label{eq:wfun1}\\
    \int_{\mathbb{R}^d} \wfun(|\vx|) |\vx|^2 d\vx &= 0. \label{eq:wfun0}
  \end{align}
\end{lemma}

\begin{proof}
  The fact that $\wfunhat$ is even and radially symmetric
  follows from well-known properties
  of the Fourier transform.
  \eqref{eq:wfun1} follows from the fact that
  \be
  \wfunhat(0) = \int_{\mathbb{R}^d} \wfun(|\vx|) d\vx = 1.
  \ee
  Since $\wfunhat$ is even, we have
  \be
  \wfunhat(k) = \wfunhat(0) + \frac{1}{2}\wfunhat''(0)k^2 + O(k^4)
  \ee
  for $k$ near the origin, and
  \be
  \wfunhat'(k) = \wfunhat''(0) k + O(k^3),
  \ee
  which leads to
  \be
  \wfunhat'(k)/k = \wfunhat''(0) + O(k^2).
  \ee
  Using the fact that $\wfunhat$ is radially symmetric, we obtain
  \be
  \Delta \wfunhat(k) = \wfunhat''(k)+\frac{d-1}{k}\wfunhat'(k).
  \ee
  Thus,
  \be
  \Delta \wfunhat(0) = d\wfunhat''(0) = 0,
  \ee
  and \eqref{eq:wfun0} follows from
  \be
  \Delta\wfunhat(0) = -\int_{\mathbb{R}^d} \wfun(|\vx|) |\vx|^2d\vx = 0.
  \ee
\end{proof}

\begin{definition}
  Suppose that $\wfun$ is a biharmonic mollifier. Then a multilevel
  splitting of the biharmonic kernel in the Fourier domain
  is given by the formula
\be
  \bhkhat(k) = 
  \bmohat_0(k) + \sum_{\ell=0}^{L-1} \bdhat_\ell(k) + \brhat_\ell(k),
  \label{eq:generalbiharks}
\ee
where
  \be
  \ba
  \bmohat_\ell(k) &= \wfunhat(k r_\ell) \Bhat(k)= \frac{\wfunhat(k r_\ell)}{k^4}, \\
  \brhat_\ell(k) &= \frac{1 - \wfunhat(k r_\ell)}{k^4}, \\
  \bdhat_\ell(k) &= \bmohat_{\ell+1}(k) -\bmohat_\ell(k),
  \ea\label{eq:generalbiharkdefs}
  \ee
and $r_\ell = 2^{-\ell}$ is the box length at level $\ell$ in the tree hierarchy.
\end{definition}

It is clear that the single-level kernel splitting for the biharmonic kernel
\be
B(r) = \Bmoll(r) + \Bres(r),
\label{eq:bh_kernelsplit2}
\ee
and its multilevel splitting in physical space 
\be
B(r)= \bmo_0(r)
 +\sum_{\ell=0}^{L-1} \bd_\ell(r) +\br_L(r)  \qquad (L \geq 1)
\label{eq:bh_multilevel_ks}
\ee
with the biharmonic \emph{difference} kernel defined by
\be
\bd_\ell(r) =\bmo_{\ell+1}(r) -\bmo_\ell(r).
\label{eq:bh_diffkernel}
\ee
follow from the application of the radial Fourier transform in Eq.~\eqref{eq:radial_ft}.
Note that \cref{def:wfunhat} implies boundedness of $\brhat_\ell(k)$ in the limit $k\to 0$, meaning that $\br_\ell$ is rapidly decaying or compactly supported.

\begin{lemma}\label{lemma:biharmeanvalueproperty}
  Suppose that $u$ is biharmonic in $\mathbb{R}^d$, i.e., $\Delta^2 u(\vx) = 0$.
  Let $\overline{u}(\vx,a)$ be the spherical mean of $u$ defined by the formula
  \be
  \overline{u}(\vx,a) = \frac{1}{\omega_{d-1} a^{d-1}}\int_{S_a(\vzero)} 
          u(\vx-\vy)dS(\vy),
  \ee
  where $S_a(\vzero)$ is the sphere of radius $a$ centered at $\vzero$
  and $\omega_{d-1}$ is the area of the unit sphere $S^{d-1}$ in $\mathbb{R}^d$.
  Then for $d\ge 2$,
  \be
  \overline{u}(\vx,a) = u(\vx) + \frac{a^2}{2d}\Delta u(\vx).
  \label{eq:biharmeanvalueproperty}
  \ee
\end{lemma}

\begin{proof}
  First, we have
  \be
  \ba
  \partial_a \overline{u}
  &= \frac{1}{\omega_{d-1}}\int_{S_1(\vzero)} u(\vx+a\vy)dS(\vy)\\
  &= \frac{1}{\omega_{d-1}}\int_{S_1(\vzero)} \vy\cdot \nabla u(\vx+a\vy)dS(\vy)\\
  &= \frac{1}{\omega_{d-1}a^{d-1}}\int_{S_a(\vzero)} \vn(\vy)\cdot \nabla u(\vx+\vy)dS(\vy)\\
  &= \frac{1}{\omega_{d-1}a^{d-1}}\int_{B_a(\vzero)} \Delta u(\vx+\vy)d\vy\\
  \ea\label{eq3.1}
  \ee
  where $B_a(\vzero)$ is the ball of radius $a$ centered at $\vzero$ and
  spherical symmetry is used to change $\vx-a\vy$ to $\vx+a\vy$ in the
  first equality. The last equality follows from the divergence theorem.
  Second, we have
  \be
  \partial_a^2 \overline{u}=-\frac{d-1}{a}\partial_a \overline{u} +
  \frac{1}{\omega_{d-1}a^{d-1}}\int_{S_a(\vzero)} \Delta u(\vx+\vy)d S(\vy),
  \label{eq3.2}
  \ee
  where the first term on the right side follows from the derivative
  of $1/a^{d-1}$, and the second term follows from the fact that the derivative
  of the integral over the ball $B_a$ is the integral over the sphere $S_a$.
  Thus,
  \be
  \partial_a^2 \overline{u}+\frac{d-1}{a}\partial_a \overline{u} =
  \frac{1}{\omega_{d-1}a^{d-1}}\int_{S_a(\vzero)} \Delta u(\vx+\vy)d S(\vy)
  = \Delta u(\vx),
  \label{eq3.3}
  \ee
  where the second equality follows from the fact that $\Delta u$ is harmonic
  and satisfies the mean value property.
  If we view $\overline{u}$ as a function of $a$ alone, 
then \eqref{eq3.3} is a second order
  ODE for $\overline{u}$ and $\Delta u(\vx)$ is a constant:
  \be
  (a^{d-1}\overline{u}')' = a^{d-1}\Delta u(\vx).
  \label{eq3.4}
  \ee
  Integrating once, and with a slight abuse of notation, we obtain
  \be
  \overline{u}'(a) = \frac{1}{d}a\Delta u(\vx) + C\frac{1}{a^{d-1}},
  \label{eq3.5}
  \ee
  where $C$ is a constant with respect to $a$.
  From \eqref{eq3.1}, we have
  \be
  \overline{u}'(0) = 0,
  \label{3.6}
  \ee
  since the volume of $B_a(\vzero)$ is $\omega_{d-1}a^d/d$.
  As $1/a^{d-1}$ blows up as $a\rightarrow 0$ for $d\ge 2$, $C$ in \eqref{eq3.5}
  must be zero.
  Thus,
  \be
  \overline{u}'(a) = \frac{1}{d}a\Delta u(\vx).
  \label{eq3.7}
  \ee
  Integrating the above equation with respect to $a$ leads to
  \be
  \overline{u}(a) = \frac{a^2}{2d}\Delta u(\vx) + C.
  \ee
  By the definition of $\overline{u}$, we have
  \be
  C = \overline{u}(0) = u(\vx).
  \label{eq3.8}
  \ee
  The result \eqref{eq:biharmeanvalueproperty} follows.
\end{proof}

\begin{theorem}\label{thm:biharksproperty}
  Suppose that the biharmonic mollifier $\wfun$ is compactly supported, i.e.,
  $\wfun(r)=0$ for $r>1$. Then
  the mollified kernel $\Bmoll$ defined in~\eqref{eq:bh_kernelsplit2}
  satisfies the property
  \be
  \Bmoll(r) = B(r), \qquad r>1.
  \label{eq:mollkernelproperty}
  \ee
\end{theorem}

\begin{proof}
  By the convolution theorem, we have
  \be
  \ba
  \Bmoll(|\v x|) &= (\wfun * B) (\v x) = \int \wfun(|\vy|) B(|\vx -\vy|) d\vy\\
  &= \int_{B_1(\vzero)} \wfun(|\vy|) B(|\vx-\vy|)d\vy\\
  &= \int_0^{1} \wfun(s) \int_{S_s(\vzero)} B(|\vx-\vy|) dS(\vy) ds,
  \ea\label{eq3.9}
  \ee
  where the third equality follows from the assumption $\wfun(r)=0$ for $r>1$,
  and the fourth equality follows from the decomposition of a ball into spherical
  shells and the fact that $\wfun$ is radial.

  Note that  \eqref{eq:biharmeanvalueproperty} holds whenever $u$ is biharmonic
  for the target $\vx$ outside $B_1(\vzero)$, which is true for the biharmonic Green's
  function $B$.
  Thus, we have
  \be
  \frac{1}{\omega_{d-1} s^{d-1}}\int_{S_s(\vzero)} B(|\vx-\vy|)dS(\vy)
  = B(|\vx|) + \frac{s^2}{2d}\Delta B(|\vx|).
  \label{eq3.10}
  \ee

  Substituting \eqref{eq3.10} into \eqref{eq3.9}, we obtain
  \be
  \ba
  \Bmoll(|\v x|) 
  &= B(|\vx|) \int_0^{1} \wfun(s)\omega_{d-1} s^{d-1}ds
  + \frac{1}{2d}\Delta B(|\vx|) \int_0^{1} \wfun(s)\omega_{d-1} s^{d+1}ds\\
  &= B(|\vx|) \int_{\mathbb{R}^d} \wfun(|\vy|)d\vy
  + \frac{1}{2d}\Delta B(|\vx|) \int_{\mathbb{R}^d} \wfun(|\vy|)|\vy|^2d\vy\\
  &= B(|\vx|),
  \ea\label{eq3.11}
  \ee
where the last equality follows from \eqref{eq:wfun1} and \eqref{eq:wfun0}.
\end{proof}

The following corollary is immediate.
\begin{corollary}
  Under the same assumption of Theorem~\ref{thm:biharksproperty},
  the residual kernel $\br_\ell(r)$ and the difference kernel $\bd_\ell(r)$ defined in \eqref{eq:bh_kernelsplit2} satisfy
  \be
  \br_\ell(r) = 0, \quad \bd_\ell(r) = 0, \qquad r>r_\ell.
  \label{eq:reskernelproperty}
  \ee
\end{corollary}

Theorem~\ref{thm:biharksproperty} assumes that $\wfun$ is compactly supported in physical space. If $\wfun$ is not compactly supported but decays rapidly, a similar procedure can be carried out to show that the residual kernel is not identically zero, but decays rapidly, with the decay rate depending on the chosen biharmonic mollifier. 
Without entering into details, the key ingredient is the spherical mean of the biharmonic 
Green's function, which is given by the formulas:
\be
\frac{1}{\omega_{d-1} s^{d-1}}\int_{S_s(\vzero)} B(|\vx-\vy|)dS(\vy)
=\frac{1}{8\pi}\left\{\begin{aligned}&(s^2+r^2)\log s -\frac{s^2+r^2}{2} + r^2, \quad r\le s,\\
&(s^2+r^2)\log r -\frac{s^2+r^2}{2} + s^2, \quad r\ge s
\end{aligned}\right.
\label{eq3.12}
\ee
for $d=2$, and
\be
\frac{1}{\omega_{d-1} s^{d-1}}\int_{S_s(\vzero)} B(|\vx-\vy|)dS(\vy)
=-\frac{1}{24\pi}\left\{\begin{aligned}&
3s+\frac{r^2}{s}, \quad r\le s,\\
&3r+\frac{s^2}{r}, \quad r\ge s,
\end{aligned}\right.
\label{eq3.13}
\ee
for $d=3$.
As an example, the Hasimoto split in~\eqref{eq:hasimoto-split}
starts from the Gaussian $\wgausshat(k) = e^{-k^2\sigma^2/4}$
and the associated biharmonic mollifier
is $\wfunhat_G(k) = \wgausshat(k) \left(1-\wgausshat(0) k^2/2\right)$.
\begin{remark}
  We could start from the Gaussian to build a compactly supported biharmonic
  mollifier as well. This involves renormalization of the biharmonic
  residual kernel in~\eqref{eq:hasimoto-split} and associated modification of the
  Fourier transform of the biharmonic mollified kernel in~\eqref{eq:hasimoto-split}.
  This kernel split is different from the Hasimoto split, since its residual kernel
  is compactly supported according to Theorem~\ref{thm:biharksproperty}. 
\end{remark}

\section{Biharmonic kernel formulas based on window function} \label{sec:prolks}
We turn now to a specific choice for the biharmonic mollifier. In order to satisfy
the key property \eqref{eq:wfunproperties}, we choose
\begin{equation}
\wfunhat(k)=\wfunchat(k) - \frac{1}{2} k \wfunchat'(k),
\label{eq:gammaB_in_phi}
\end{equation}
where $\wfunchat(k)$ is the Fourier transform of a window function $\wfunc(x)$, which is an even function with $\wfunchat(0)=1$.
We have introduced two such window functions in \eqref{eq:wfunchat_prol_G}-\eqref{eq:wfunc_prol_G}, the prolate function $\wprol^c$ and the Gaussian $\wgauss$.
The biharmonic mollifier based on $\wprol^c$ will be compactly supported and hence will yield a compactly supported residual kernel, while the one based on $\wgauss$ will be rapidly decaying, as discussed after 
\cref{thm:biharksproperty}.

\begin{remark}
  Unlike Gaussians whose derivatives are the product of the same Gaussian
  with a polynomial, the derivatives of prolates do not have such property.
  Thus, even though 
  $\wfunhat(k)=\wfunchat(k)\left(1 - \frac{1}{2} \wfunchat''(0)k^2\right)$
  will lead to the same biharmonic mollifier $\wfunhat$  when $\wfunchat$ is a Gaussian,
  it provides a slightly
  different biharmonic kernel split from \eqref{eq:gammaB_in_phi} when $\wfunchat$ is
  a prolate.
  The performance of these two biharmonic kernel splittings using prolates
  is very close. In the remainder of this paper, we focus on \eqref{eq:gammaB_in_phi}.
\end{remark}

\subsection{The residual biharmonic kernel}
\label{sec:biharmonic-split}
With the mollified biharmonic kernel defined in the Fourier domain, the residual kernel in physical space 
can be computed numerically. The procedure for this was introduced in \cite{jiang2025cpam}, and 
we outline it for the case of the biharmonic in \cref{app:num_fourier_split}. However, it will
be convenient to have an explicit expression for it in three dimensions, so we now compute this
 using the biharmonic mollifier in \eqref{eq:gammaB_in_phi}.
We first note that the mollifier's \emph{one-dimensional} inverse Fourier transform is 
\begin{align}
  \gamma(x)=\frac{1}{\pi}\int_0^\infty \cos(kx) \wfunhat(k)dk =
  \frac{3}{2}\wfunc(x) + \frac{1}{2} x \wfunc'(x).
  \label{eq:biharmonic_screen_R}
\end{align}

\begin{theorem} \label{thm:Bres}
  Suppose that $\Bmollhat(k)=\bmohat_0(k)$ is given by \eqref{eq:generalbiharks},
  that $\wfunhat(k)$ is given by \eqref{eq:gammaB_in_phi} and that $\gamma(x)$
  is given by \eqref{eq:biharmonic_screen_R}. Then in $\R^3$,
  \begin{align}
    \Bres(r)=-\frac{1}{8\pi}r-\Bmoll(r)
    = -\frac{1}{8\pi}\left(r - 2 r \int_0^r \wfunc(s) \dif s
    + 2  \int_0^r s \wfunc(s) \dif s
    - 2  \int_0^\infty s \wfunc(s) \dif s\right). 
    \label{eq:BR_gen}
  \end{align}
\end{theorem}

\begin{proof}
In order to compute $\Bres(r)=-\frac{1}{8\pi}r-\Bmoll(r)$, we need the inverse
transform
\begin{align}
    \Bmoll(r) &= \frac{1}{2\pi^2} \int_0^\infty \frac{\sin(kr)}{kr} \Bmollhat(k) 
      k^2 \dif k
              = \frac{1}{2\pi^2 r} \int_0^\infty \frac{\sin(kr)}{k^3}
              \widehat\gamma(k)
              \dif k
              = -\frac{1}{2\pi r} I(r),
\end{align}
where
\begin{align}
  I(r) = -\frac{1}{\pi} \int_0^\infty \frac{\sin(kr)}{k^3}
              \widehat\gamma(k)
              \dif k .
\end{align}
To compute $I(r)$, we differentiate three times, noting that this leads to 
the one-dimensional inverse Fourier transform,
\begin{align}
  I'''(r) &= \frac {1}{\pi}\int_0^\infty  \cos(kr)
           \widehat\gamma(k)
           \dif k = \gamma(r) = \frac{3}{2}\wfunc(r) + \frac{1}{2} r \wfunc'(r).
\end{align}
Integrating by parts three times, we have 
\begin{align}
  I''(r) &= \int_0^r \wfunc(s) \dif s + \frac{1}{2} r \wfunc(r) + C_1, \\
  I'(r) &=  r \int_0^r \wfunc(s) \dif s - \frac{1}{2} \int_0^r s \wfunc(s) \dif s + C_1 r + C_2, \\
  I(r) &= \frac{1}{2} \del{ r^2 \int_0^r \wfunc(s) \dif s
         -  r \int_0^r s \wfunc(s) \dif s + C_1 r^2 + C_2 r + C_3 }. 
\end{align}
Thus
\begin{align}
  \Bmoll(r) &= -\frac{1}{8\pi}\left(2r \int_0^r \wfunc(s) \dif s
         - 2 \int_0^r s \wfunc(s) \dif s + C_1 r + C_2 + \frac{C_3}{r}\right) .
\end{align}
We determine the constants by requiring that $B_M(r)$ be smooth and bounded as $r\to 0$,
and
$\lim_{r\rightarrow \infty}(B(r)-B_M(r))=0$. This yields 
\begin{align}
  \Bmoll(r) &= -\frac{1}{8\pi}\left(2 r \int_0^r \wfunc(s) \dif s
              -2 \int_0^r s \wfunc(s) \dif s
              +2 \int_0^\infty s \wfunc(s) \dif s\right)
              \label{eq:Bmoll_gen}
\end{align}
and the desired result \eqref{eq:BR_gen}. 
\end{proof}

\begin{corollary}
  The first three derivatives of the residual biharmonic kernel $\Bres(r)$ 
are
\be
\begin{split}
    \Bres'(r) &= -\frac{1}{8\pi}\left(1 - 2 \int_0^r \wfunc(s) \dif s\right)
    = -\frac{1}{8\pi}\Phi(r), \\
    \Bres''(r) &= \frac{1}{4\pi} \wfunc(r) , \\
    \Bres'''(r) &= \frac{1}{4\pi} \wfunc'(r),
\end{split} \label{eq:BR_ders} 
\ee
with $\Phi(r)$ defined by the formula
\begin{align}
    \Phi(r) = 2\int_r^\infty \wfunc(t) \dif t.
    \label{eq:def_erfc_like_fcn}
\end{align}
\end{corollary}

The result follows immediately from \cref{thm:Bres} and will be needed in \cref{sec:stokes-split-deriv}
to derive the Stokeslet and stresslet splittings.

\section{Summation of the Stokes kernels}
In this section, we give the formulas for the split of the Stokeslet, the stresslet and the rotlet expressed with the window function $\wfunc$.

The split of the Stokeslet and the stresslet can be obtained by differentiating the split of the biharmonic kernel introduced above. 
The rotlet will instead be based on the split of the harmonic kernel, 
\begin{equation}
\Hres(r)=\frac{\Phi(r)}{4\pi r}, \qquad \Hmollhat(k)=\frac{1}{k^2}\wfunchat(k), 
\label{eq:harmonic_split}
\end{equation}
with $\Phi$ defined from $\wfunc$ in \eqref{eq:def_erfc_like_fcn}.
For the Gaussian window function $\wgauss$ (\ref{eq:wfunc_prol_G}), this is the classical Ewald split introduced in \eqref{eq:ewald-split}.

\subsection{The self-interaction correction}

Suppose that the target $\vx_\tind$ is also a source location and that we
intend to compute
\[
u(\vx_\tind) = \sum_{\substack{\sind=1 \\ \sind\neq\tind}}^N K(\vx_\tind-\vx_\sind) \rho_\sind, \quad \tind=1,\dots,N,
\]
without any {\em self-interaction}. 
For a general splitting, however, both Ewald and DMK-type algorithms compute
\[
\tilde{u}(\vx_\tind) = \sum_{\sind=1}^N K_M(\vx_\tind-\vx_\sind) \rho_\sind 
\ +\ \sum_{\substack{\sind=1\\ \sind \neq \tind}}^N K_R(\vx_\tind-\vx_\sind) \rho_\sind 
\quad \tind=1,\dots,N.
\]
As some contribution from the self interaction is included through the mollified kernel, we must add to $\tilde{u}(\vx_\tind)$ a correction term 
\be
u(\vx_\tind) = \tilde{u}(\vx_\tind)+ u_{\rm self}(\vx_\tind)
\ee
where
\be
u_{\rm self}(\vx_\tind) = -
\left(\lim_{r\rightarrow 0} K_M(r)\right)\rho_{\tind}.
\ee
We describe these corrections (when needed) in the next section, as we derive the full
Stokes kernels. 

\subsection{The Stokes kernels}
\label{sec:stokes-split}

Here,
we collect the formulas for the mollified and residual kernels for 
the Stokeslet, stresslet and rotlet in three dimensions. The derivations can be found in 
\cref{sec:stokes-split-deriv}, with residual kernels for two dimensions computed using the procedure of \cref{app:num_fourier_split}.

The {\em mollified Stokeslet} in Fourier space is given by
\begin{align}
    \widehat S_{M,jl}(\v k) = \left(k^2 \delta_{jl} - k_j k_l\right) \Bmollhat(k),
  \quad k=\abs{\v k}.
  \label{eq:moll_stokeslet}
\end{align}

The {\em residual Stokeslet} is given by
\begin{align}
    S_{R,jl}(\vx)
    &=   \RSdiag(r) \frac{\delta_{jl}}{8\pi r}
      +   \RSoffd(r)  \frac{x_jx_l}{8\pi r^3},
      \label{eq:res_stokeslet}
\end{align}
where 
\be
  \RSdiag(r) = 
               \Phi(r) - 2 r\wfunc(r),  \qquad
  \RSoffd(r) = 
               \Phi(r) + 2 r\wfunc(r).
               \label{eq:Soffdiag}
\ee
with $\Phi$ defined from $\wfunc$ in \eqref{eq:def_erfc_like_fcn}.
The rate of decay of these functions will depend on the choice of $\wfunc(r)$.
The self interaction term is given by
\begin{align}
\uself(\vx_\tind) = -\lim_{|\v x|\to 0} S_M(\v x) \rho_\tind=
-\frac{1}{2\pi} \wfunc(0) \rho_\tind.
\label{eq:S_self}
\end{align}

The {\em mollified stresslet} in Fourier space is given by
\begin{align}
    \widehat T_{M,jlm}(\v k) &=
 i \left[(k_m \delta_{jl} +k_j \delta_{lm} +k_l  \delta_{mj}) k^2   -2
  k_j k_l k_m \right]  
                             \Bmollhat(\vk) .
                             \label{eq:moll_stresslet}
\end{align}

The {\em residual stresslet} by
\begin{align}
  T_{R,jlm}(\vx) = \frac{\delta_{jl}x_m + \delta_{lm}x_j + \delta_{mj}x_l}{8\pi r^3} \, \RTdiag(r)
  - \frac{3}{4\pi}\frac{x_jx_lx_m}{r^5} \, \RToffd(r),
  \label{eq:res_stresslet}
\end{align}
where 
\be
  \RTdiag(r) = 
               -2 r^2 \wfunc'(r) ,
\quad                
  \RToffd(r) = 
               \Phi(r) + 2r \wfunc(r)  - \tfrac{2}{3} r^2 \wfunc'(r).
               \label{eq:T-offd}
\ee
The self interaction term for the stresslet vanishes: 
$\uself(\vx_\tind) = 0$.

The {\em mollified rotlet} in Fourier space is given by
\begin{align}
  \widehat \Omega_{M,jl}(\v k) = - \frac{1}{2} i \epsilon_{jlm} k_m \widehat H_M(\v k).
  \label{eq:moll_rotlet}
\end{align}

The {\em residual rotlet} follows directly from \eqref{eq:rotlet_def_deriv} together with \eqref{eq:harmonic_split} and is given by
\begin{align}
  \begin{split}
    \Omega_{R,jl}(\vx) = 
    \frac{1}{8\pi} \epsilon_{jlm} \frac{x_m}{r^3} \ROoffd(r),
  \end{split}
  \label{eq:res_rotlet}
\end{align}
with
\begin{align}
  \ROoffd(r) =\Phi(r) + 2r\wfunc(r).
  \label{eq:ROoffd}
\end{align}
The self-interaction term for the rotlet vanishes: $\uself(\vx_\tind) = 0$.

\begin{remark}To scale the Stokes kernels with a length scale $\nu$, all real space functions ($\RSdiag$, $\RSoffd$, $\RTdiag$, $\RToffd$, $\ROoffd$) should be evaluated at $r/\nu$, all else is left the same.  The scaling in Fourier space enters in $\Bmollhat$, through $\wfunhat(k\nu)$ as given in \eqref{eq:generalbiharkdefs}.
\end{remark}

\begin{remark}
  Assuming a smooth $\wfunc(r)$, the real-space functions ($\RSdiag$, $\RSoffd$, $\RTdiag$,
  $\RToffd$, $\ROoffd$) used in constructing the various Stokes kernels above are all smooth
  functions of $r$, and can be accurately represented and rapidly
  evaluated using a polynomial approximation which is precomputed for each $\wfunc$.
\end{remark}
\subsection{The total sum}
\label{sec:totsum}
Using a one-level split of the kernels, and inserting into the sum \eqref{eq:fs_sum} for a free-space problem, or \eqref{eq:periodic_sum} for a periodic problem, we obtain
\begin{equation}
    u(\xb)=\ulocal(\xb)+\ufar(\xb)+\uself(\xb), \qquad \tind=1,\ldots,N.
\end{equation}

For the free space problem we have 
\begin{align}
    \ulocal(\vx_\tind) &= \sum_{\substack{\sind=1\\\sind\ne \tind}}^N  \Kres(\vx_\tind-\vx_\sind) \rho_\sind, 
                     \label{eq:fs_ewald_local}
\\
  \ufar(\vx_\tind) &= \sum_{\sind=1}^N 
    \cbr{ \frac{1}{(2\pi)^3} \int_{\R^3} \Kmollhat(\vk) e^{i\vk\cdot(\vx_{\tind}-\vx_\sind)} d\vk + K^\truncrad_{\text{corr}}}
        \rho_\sind. \label{eq:fs_ewald_far}
\end{align}
Note that the mollified biharmonic kernel in \eqref{eq:generalbiharks} is singular at $\vk = \vzero$, as it contains the factor $\Bhat(|\vk|)=1/|\vk|^4$. As this carries through to the Stokeslet and the stresslet, integrals in the Fourier domain cannot be accurately approximated using a simple trapezoidal
rule. For free-space problems, we replace the biharmonic kernel $\Bhat$ in \eqref{eq:moll_stokeslet} and \eqref{eq:moll_stresslet} by the truncated
alternative $\widehat B^\truncrad(k)$ from \eqref{eq:Btrunc_hat}.  That is, we will
use the \emph{truncated mollified kernel} given by
\be
\Bhat^\truncrad_{M}(k) = \widehat B^\truncrad(k) \wfunhat(k),
\label{eq:Bmollhat_trunc}
\ee
in the definition of $\widehat S_{M}$ and $\widehat T_{M}$, and then discretize the integral in \eqref{eq:fs_ewald_far} using the trapezoidal rule. 
Similarly, we replace $\widehat H(k)$ in \eqref{eq:moll_rotlet} by the truncated alternative $\widehat H^\truncrad(k)$, as also used in \cite{jiang2025cpam}.
The use of a truncated kernel for the far-field evaluation free-space can introduce a constant term which must be corrected for using the term $K^\truncrad_{\text{corr}}$.
This is zero for the stresslet and rotlet, but for the Stokeslet we have (see \eqref{eq:msdiag} in \cref{app:num_fourier_split})
\begin{align}
 K^\truncrad_{\text{corr}} = S^\truncrad_{\text{corr}} =
  \begin{cases}
    \frac{1}{4\pi}(1-\log\truncrad), & d=2, \\
    \frac{1}{4\pi\truncrad}, & d=3.
  \end{cases}
\end{align}

For the periodic problem on the unit cube, we get
\begin{align}
    \ulocal(\vx_\tind) &= \sum_{\sind=1}^N \sum_{\v p\in\mathbb Z^3}^* 
    \Kres(\vx_\tind-\vx_\sind + \v p) 
    \rho_\sind,
    \label{eq:ewald_local}
\end{align}
where the $(*)$ notation signifies that the $(\v p=0, \tind=\sind)$ term is ignored. 
The far-field part is evaluated in Fourier space using the Poisson summation formula 
for a unit cube
\begin{align}
\ufar(\vx_\tind) &= \sum_{\sind=1}^N \sum_{\v p\in\mathbb Z^3} \Kmoll(\vx_\tind-\vx_\sind + \v p) \rho_\sind
= \sum_{\sind=1}^N \sum_{\substack{\vk\in\mathcal K\\ \vk \ne 0}} \Kmollhat(\vk) e^{i\vk\cdot(\vx_\tind-\vx_\sind)}
  \rho_\sind
  + u^{(0)}(\vx_\tind).
\label{eq:ewald_far}
\end{align}
with $\mathcal K = \{2\pi\kappa : \kappa\in\mathbb Z^3\}$.
 The term $u^{(0)}$ is a constant corresponding to the term $\vk=0$ omitted from the Fourier sum. To obtain a zero mean flow for $u(\vx_{\tind})$ for the Stokeslet and the rotlet, this term can be set to zero. For the stresslet however, refering back to the specific notation in \eqref{eq:stokes_stress_sum}, we need to set 
\begin{equation}
     u^{(0)}(\vx_\tind)=- \sum_{\sind=1}^N
     (\v x_{\tind}-\v x_{\sind})\, ({\bf f}_{\sind} \cdot {\bf n}_{\sind}). 
\end{equation}
See \cite{AfKlinteberg2014a} and \cite[Appendix~A]{bagge_fast_2023} for details. 

Introducing the multilevel split of the Stokes kernels, following the multilevel split of the biharmonic kernel in 
 \eqref{eq:generalbiharks}, we obtain a decompostion of the kernel of the form \eqref{eq:multilevel_ks}. Once inserted in the free-space \eqref{eq:fs_sum} or periodic \eqref{eq:periodic_sum} sum, we get
\begin{align}
        u(\vx_{\beta}) = u_0^{\rm far}(\vx_{\beta})+ \sum_{\ell=0}^{L-1} u_{\ell}^{\rm diff}(\vx_{\beta}) +u_L^{\rm local}(\vx_{\beta}) + \uself(\vx_\beta),
\label{eq:three_part_sum}
    \end{align}
where the subindices indicate the different scalings introduced. 
In the DMK method, quadrature in the Fourier domain (as explained in section \ref{sec:trapquad}) is used for the difference kernels, that are smooth also in Fourier space. The diagonal translation property of this representation, together with the localization of both the residual kernel and the difference kernel at different scales, are key properties utilized when constructing the DMK method.  
The maximum level $L$ in the sum will vary for different target points, as tied to an adaptive level restricted oct-tree.

\section{The DMK method} \label{sec:fast}

In this section, we first summarize the DMK framework and
refer the readers to \cite{jiang2025cpam} for more details, where the free-space case was considered. We then discuss the modifications needed for periodic boundary conditions on a cube. 

\subsection{Summary of the DMK framework}
In brief,
DMK is a hierarchical fast algorithm that uses Fourier convolution to compute
interactions with near neighbors at each 
spatial scale with a {\em short} Fourier transform of length $O(\log^d(1/\epsilon))$,
{\em independent of level}. Given a (correctly chosen) telescoping series of the form
\eqref{eq:multilevel_ks}, interactions with the $M_0$ kernel are carried out at level $0$,
and interactions with the difference kernel $\kd_\ell$ are carried out at level $\ell$. 
Once a leaf node at level $L$ is reached, residual kernel interactions are computed
with nearest neighbors alone. 
A critical aspect of the method is that since the mollified potentials are smoother
and smoother at coarser and coarser levels, one can replace the sources with 
an equivalent (and much smaller) set of ``proxy" charges at each level. Likewise,
the mollified potential can be sampled on a tensor product grid with a fixed 
set of $p^d$ points on every box in the tree hierarchy. 
Finally, the depth of the telescoping series can vary based on 
the extent of clustering in the sources using an adaptive tree.

\paragraph{Step 1: Tree construction:}

(a) Sort the sources and targets hierarchically using an 
adaptive quad-tree (2D) or oct-tree (3D) until there are fewer 
than $s$ sources in a leaf node.  Some definitions are in order:
the colleagues of a box $\unitbox$ are the boxes at the same refinement level which share 
a boundary point with $\unitbox$ (including $\unitbox$ itself.) 
The coarse neighbors of $\unitbox$ are leaf nodes 
at the level of $\unitbox$’s parent which share a boundary point with $\unitbox$. The fine neighbors of $\unitbox$ are leaf nodes one level finer than $\unitbox$ which share a boundary. The union of the colleagues, coarse neighbors, and fine neighbors of $\unitbox$ are referred to as $\unitbox$’s neighbors.

(b) Ensure that the tree is level-restricted: that is,
two leaf nodes which share a boundary point must be no more than one refinement level 
apart (see Fig.~\ref{fig:trees}). This can result in slightly different data structures
for the free-space and periodic cases. See
\cite{ethridge2001sisc,sundar2008sisc} for further details.

(c) Superimpose a tensor product proxy grid of $p^d$ scaled Chebyshev nodes
on each box in the tree hierarchy.
\begin{figure}[ht]
  \centering
  \includegraphics[height=50mm]{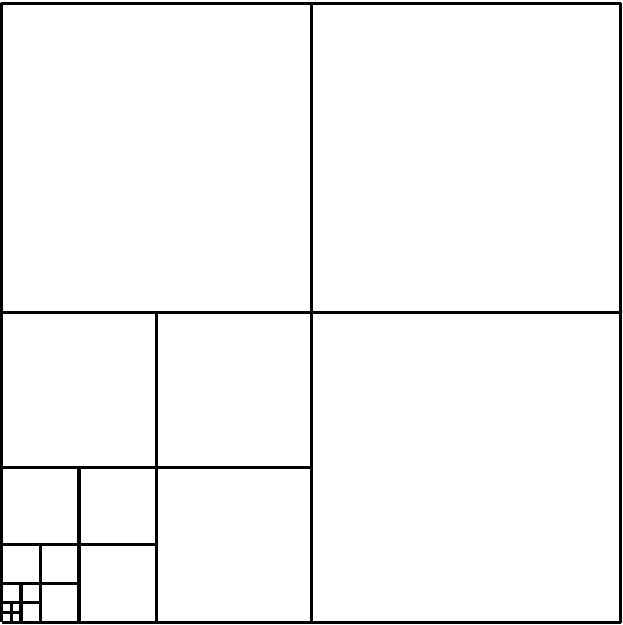}
  \hspace{6mm}
  \includegraphics[height=50mm]{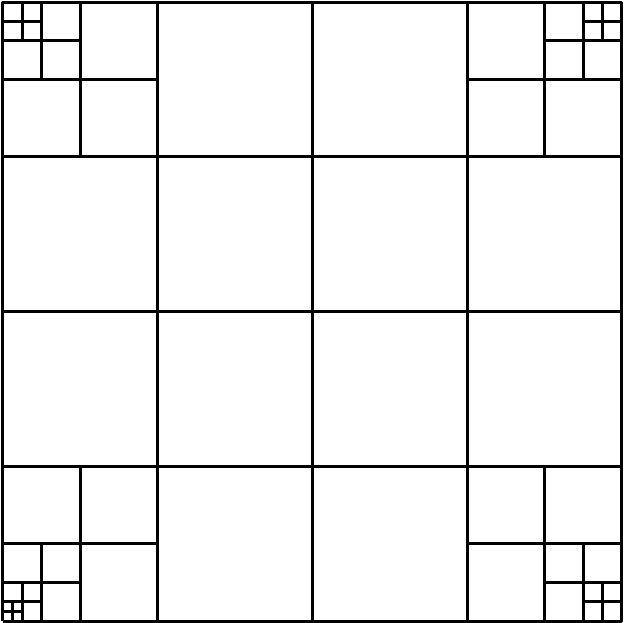}

  \caption{Level-restricted adaptive trees in two dimensions.
On the left is an example that satisfies the level restriction for free space
interactions. In the periodic setting however, the fine grid boxes are neighbors
of coarse periodic images. On the right, we illustrate the additional refinements
needed to satisfy the level-restriction property in the periodic case.}
  \label{fig:trees}
\end{figure}

\paragraph{Step 2: Upward Pass:}

(a) Construct an equivalent set of $p^d$ tensor-product {\em proxy} sources
on each leaf node. Assuming the leaf node is at level $\ell$,
the proxy charges are chosen to induce the same field as the original
sources to the desired precision $\epsilon$ assuming the interaction kernel
is given by the (locally smooth) difference kernel $\kd_\ell(\v x)$.

(b) In a fine-to-coarse sweep, recursively 
merge the proxy sources in all of the $2^d$ children into
a single proxy grid at the parent level.
[{\em  For nonoscillatory kernels such as
the Coulomb or biharmonic kernel, $p$ can be held constant at each level with
no loss of precision, resulting in a significant form of data compression}.] 

\paragraph{Step 3: Downward Pass:}

(a)
At level $\ell=0$,  compute the convolution of the proxy grid in the root box with the
truncated mollified kernel $\km_0^\truncrad(k)$ from 
\eqref{eq:Bmollhat_trunc} (for the free-space case, see \cref{sec:periodicity} for the periodic case). The output points can be chosen to be the same 
tensor product proxy grid, and we denote the values stores as the {\em proxy potentials}.
Once available, interpolate the proxy potentials to the children at level $1$.

(b)
In a coarse-to-fine sweep, 
for each box $\unitbox$ at level $\ell$, the 
convolution of the proxy grid in each of $\unitbox$'s colleagues 
 with the difference kernel $\kd_\ell$ is 
computed in the Fourier domain and added to the proxy potentials in $\unitbox$
using translation operators. Once the proxy potentials are available, they are 
interpolated to $\unitbox$'s children.
We refer the reader to \cite{jiang2025cpam} for technical details, except to 
note that work per box is independent of the level, combining the observation
in \cref{rmk:scaling} with the careful construction of the multilevel splitting
in \cref{sec:biharmonicks}.

\paragraph{Step 4: Residual Interactions:}

Once a leaf node is reached, the residual kernel interactions governed
by $\kr_\ell(\vx)$ are localized to 
near neighbor boxes and computed directly. 

\begin{remark}
For readers familiar with the original DMK algorithm, 
we have made some minor changes. First, we use Fourier convolution only for 
colleague interactions. As a result, the direct interactions for a box now match
the direct interaction component in the fast multipole method (FMM), although
with a different kernel, so that we can exploit mature MPI implementations of 
the FMM, such as \cite{malhotra_pvfmm_2015}. Second, the Fourier scaling of the stresslet represents a contraction from 9 to 3 components, so it is more efficient to apply it to the outgoing expansion rather than the incoming expansions, as was the case in the original algorithm.
\end{remark}

\begin{remark}
Unlike Ewald summation, after tree construction, DMK is a linear scaling
algorithm {\em that does not depend on the FFT for its performance }. All the Fourier integrals are computed using a direct implementation of the discrete
Fourier transform since the expansion lengths are too small to benefit from
a fast algorithm. The speed comes, in essence, from hierarchical compression
as in the FMM.
\end{remark}

\subsection{Periodic boundary conditions in the DMK framework}
\label{sec:periodicity}
It is straightforward to modify the DMK algorithm in order to 
impose periodic boundary conditions on a cube. 
First, instead of using the windowed kernel at the root level, we use the (periodic)
Fourier series corresponding to the mollified kernel, see the discussion around equations \eqref{eq:fs_ewald_far} and
\eqref{eq:ewald_far}.

Second, we change the definition of neighbors to include periodic images of boxes in the unit cell
which share a boundary point. Hence, the changes for the difference and residual kernel interactions are trivial:
one simply replaces the free-space near neighbor list with the periodic near neighbor 
list.

From the discussion in \cref{rmk:scaling} regarding the difference kernel, the number of Fourier modes is $O(1)$ also for the truncated mollified kernel used in free-space since the bandlimit of $\kmhat_0(k)$ is $O(1)$ and 
the targets are restricted to the unit cube.
The number of Fourier modes needed are however typically smaller in the periodic case, as the discrete sum in this case arises directly from the Poisson summation formula, and not from the discrerization of an integral.   

\section{Error control and parameter selection}
\label{sec:par_selection}

One critical parameter in DMK is the order $p$ of the tensor product proxy grids
defined on each box. $p$ must be sufficiently large that it accurately
interpolates the proxy potentials in the downward pass and accurately anterpolates
the difference kernels in the upward pass.
The other critical parameters are the bandlimit $\Kmax$, and the choice of $N_1$ --
the number of discretization nodes in each dimension
needed for computing the Fourier integrals at each level. 
These integrals are 
of the form \eqref{ewalds3}, with 
$\Kdiffhat(\vk)=\widehat{K}_{D_l}(\vk)$
at level $\ell$
and $\widehat{\rho}(\vk)$ defined to be the Fourier transform of the proxy grid
sources in the box. 
The integrals are truncated at $k_i \in [-\Kmax/r_{l+1}, \Kmax/r_{l+1}]$ and the grid spacing is $h_l=2\pi/(3 r_l)= \pi / (3 r_{l+1})$. 
When using the prolate window function $\Kmax = c$ (and referred to as such in the 
experiments below).

To study how the error depends on $c$ and $p$ for the Stokes kernels,
we run parameter sweeps for each kernel on a system of 5000 uniformly
random points and strengths in the unit cube. The relative $l_2$ error
$E(c,p)$ of DMK, compared to a direct sum, is computed for every
$(c,p) \in \{\frac{\pi}{3}, \dots 50\cdot\frac{\pi}{3}\} \times
\{1,\dots,50\}$. An example of what the results look like is found in
Fig.~\ref{fig:stokeslet_param_sweep}. From the accumulated data, we find
the optimal (i.e. smallest) $(c,p)$ that satisfy a tolerance
$\epsilon=10^{\{-3,-6,-9,-12\}}$. These are listed in
\Cref{tab:params}, where optimal ($\sigma,p)$ values for the Gaussian
split is also found.

\begin{figure}[ht]
  \centering
  \begin{subfigure}{.49\textwidth}
    \centering
    \includegraphics[height=0.7\textwidth]{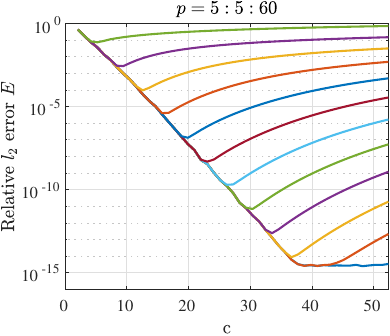}
    \caption{Error curves for $p$ held fixed with $c$ varying.}
    \label{fig:stokeslet_param_sweep_c}
  \end{subfigure}
  \begin{subfigure}{.49\textwidth}
    \centering
    \includegraphics[height=0.7\textwidth]{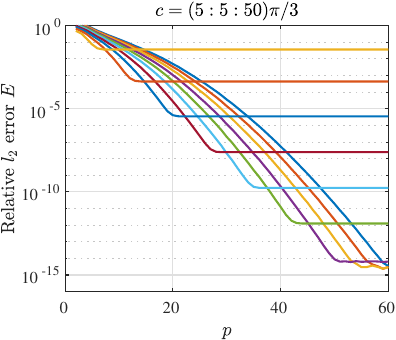}
    \caption{Error curve for $c$ held fixed with $p$ varying.}
    \label{fig:stokeslet_param_sweep_p}
  \end{subfigure}
  \caption{Parameter sweeps for the Stokeslet. In
    (\subref{fig:stokeslet_param_sweep_c}), the error decays with $c$
    until interpolation errors start dominating due to $p$ being too
    small. In (\subref{fig:stokeslet_param_sweep_p}), the error decays
    with $p$ until it hits the floor where the interpolation is fully
    resolved and the error equals the truncation error set by $c$. 
    }
  \label{fig:stokeslet_param_sweep}
\end{figure}

In order to find parameters that satisfy any given tolerance
$\epsilon$, we extract from our data $(c,p,E)$ triplets where $c$ and
$E$ are given by the lowest attainable error for a given $p$.  By
linear regression we can then fit straight lines to the pairs
$(c, \log_{10} E)$ and $(p, c)$, which ultimately allow us to find our
desired parameters by letting $E=\epsilon$. This is used in our
implementation for automatic parameter selection.

\begin{table}[ht]
  \centering
  \textbf{Stokeslet}\vspace{3pt}
  \begin{tabularx}{\textwidth}{Y|YYYY|YYYY}
\hline
$\epsilon$ & $\frac{3}{\pi}c$ & $p$ & $N_1$ & $N_{\rm{per}}$ & $\frac{6}{\pi}\sigma^{-2}$ & $p^{\rm G}$ &  $N_1^{\rm G}$ & $N_{\rm{per}}^{\rm G}$ \\
\hline
$10^{ -3}$ &  10 &  12 &  19 &   3 &  14 &  14 &   27 &   5 \\
$10^{ -6}$ &  17 &  23 &  33 &   5 &  25 &  28 &   49 &   9 \\
$10^{ -9}$ &  25 &  33 &  49 &   9 &  38 &  43 &   75 &  13 \\
$10^{-12}$ &  31 &  44 &  61 &  11 &  50 &  58 &   99 &  17 \\

\hline
\end{tabularx}
  \\
  \vspace{1em}
  \textbf{Stresslet}\vspace{3pt}
  \begin{tabularx}{\textwidth}{Y|YYYY|YYYY}
\hline
$\epsilon$ & $\frac{3}{\pi}c$ & $p$ & $N_1$ & $N_{\rm{per}}$ & $\frac{6}{\pi}\sigma^{-2}$ & $p^{\rm G}$ &  $N_1^{\rm G}$ & $N_{\rm{per}}^{\rm G}$ \\
\hline
$10^{ -3}$ &   9 &  11 &  17 &   3 &  11 &  12 &   21 &   3 \\
$10^{ -6}$ &  17 &  22 &  33 &   5 &  26 &  28 &   51 &   9 \\
$10^{ -9}$ &  25 &  33 &  49 &   9 &  40 &  44 &   79 &  13 \\
$10^{-12}$ &  32 &  44 &  63 &  11 &  54 &  60 &  107 &  17 \\

\hline
\end{tabularx}
  \\
  \vspace{1em}
  \textbf{Rotlet}\vspace{3pt}
\begin{tabularx}{\textwidth}{Y|YYYY|YYYY}
\hline
$\epsilon$ & $\frac{3}{\pi}c$ & $p$ & $N_1$ & $N_{\rm{per}}$ & $\frac{6}{\pi}\sigma^{-2}$ & $p^{\rm G}$ &  $N_1^{\rm G}$ & $N_{\rm{per}}^{\rm G}$ \\
\hline
$10^{ -3}$ &   7 &   8 &  13 &   3 &   9 &   9 &   17 &   3 \\
$10^{ -6}$ &  14 &  18 &  27 &   5 &  21 &  23 &   41 &   7 \\
$10^{ -9}$ &  20 &  28 &  39 &   7 &  34 &  38 &   67 &  11 \\
$10^{-12}$ &  27 &  39 &  53 &   9 &  47 &  53 &   93 &  15 \\

\hline
\end{tabularx}
  \caption{Parameters used for achieving a relative error $\epsilon$
    in DMK for the Stokes kernels. The left values are for a PSWF
    split with band limit $c$, while the right values with superscript
    $G$ are for a Gaussian split with shape parameter $\sigma$. The
    value $p$ is the polynomial approximation order, $N_1$ is the
    total number of Fourier modes in each dimension, and
    $N_{\rm{per}}$ is the number of modes in each direction at the top
    level when the sum is periodic.}
  \label{tab:params}
\end{table}

\section{Numerical results} \label{sec:results}

We have implemented the DMK algorithm for the Stokes kernels, and
the code was compiled using the Intel compiler and linked with the
Intel MKL library, with all experiments run in single-threaded mode on a 3.30GHz
Intel(R) Xeon(R) Gold 6234 CPU. It is available as open source software \cite{DMKcode}.
We report timings for the fully adaptive case alone, omitting 
the fast Ewald version (which is generally faster than DMK
when the particles are uniformly distributed).
Since the performance of the periodic and free-space codes are similar, we only report
on the free-space case. A MATLAB version is also available \cite{DMK_matlab}, 
to facilitate experimentation and further algorithmic development. It
makes use of the excellent Chebfun package \cite{chebfun}
for most steps involving polynomial expansions.

In our implementation, we set the subdivision
parameter $n_s$, the maximum number of points in a leaf node, to
$120$, $240$, $360$, $480$ for error tolerance
$\eps = 10^{-3}, 10^{-6}, 10^{-9}, 10^{-12}$ for two dimensions, and
to $600$, $1200$, $2000$, $3000$ for three dimensions.
Figure~\ref{fig:linearscaling} shows the total runtime (in seconds)
for points distributed on a perturbed circle in two dimensions and on a perturbed
sphere in three dimensions, with the total number of source and additional
target points given by
\( N_S = N_T = \num{100000},\ \num{250000},\ \num{500000},\ \num{1000000},\ \num{2000000},\ \num{4000000} \). The number of levels in the adaptive octree ranges from $5$ to $11$
in three dimensions.

\begin{figure}[ht]
  \centering
  \includegraphics[height=48mm]{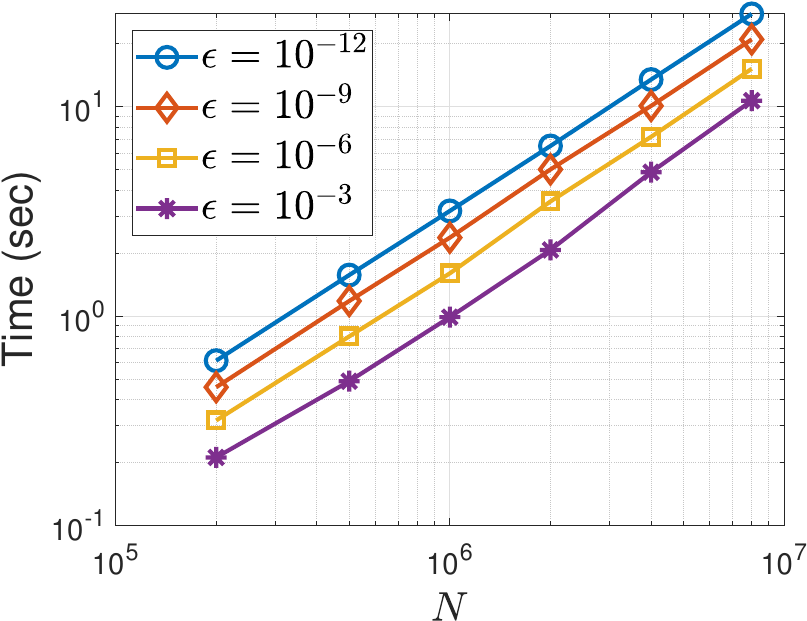}
  \includegraphics[height=48mm]{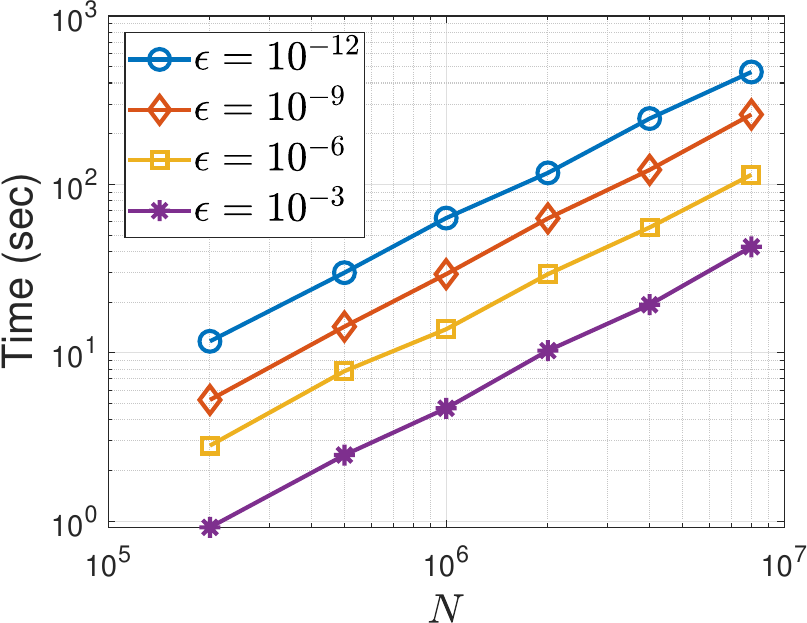}

  \vspace{4mm}
  
  \includegraphics[height=48mm]{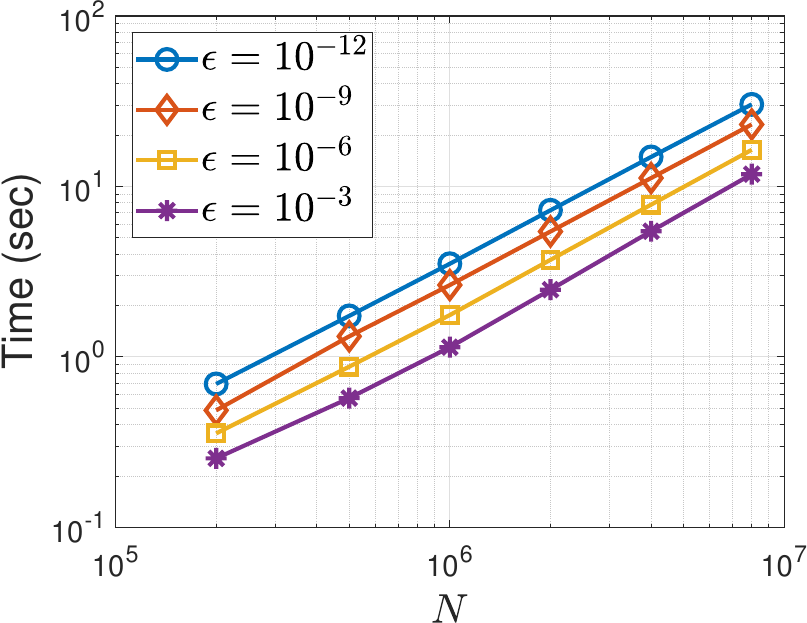}
  \includegraphics[height=48mm]{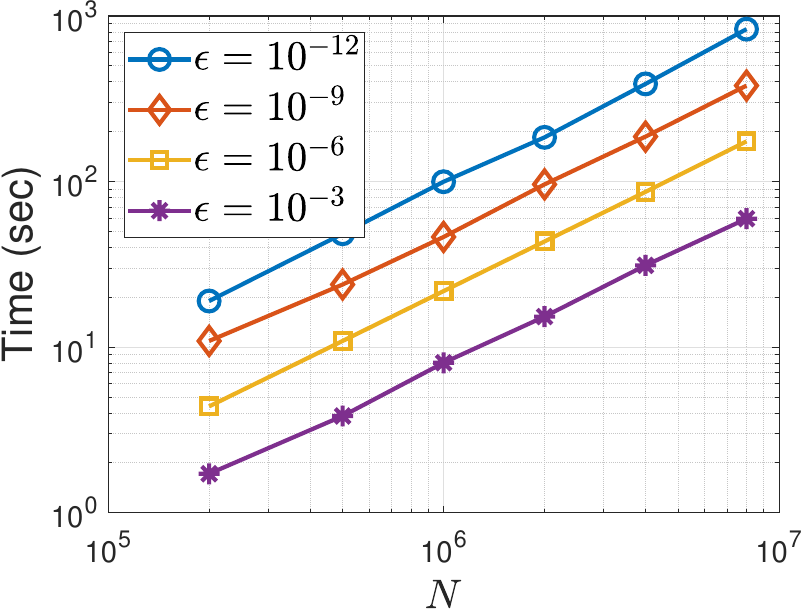}

  \caption{Linear scaling of DMK for the Stokeslet and stresslet in two and
    three dimensions. Top: Stokeslet; bottom: stresslet. Left: 2D; right: 3D.}
  \label{fig:linearscaling}
\end{figure}

In Figure~\ref{fig:throughput}, we show the average throughput of DMK
and FMM. We use \texttt{fmm2d} from \cite{fmm2dlib} and \texttt{fmm3d}
from \cite{fmm3dlib}. The Stokes FMM implemented in \cite{fmm3dlib}
uses the algorithm in \cite{Tornberg2008} which in turn calls the harmonic
FMM in \cite{cheng1999jcp}. Both DMK and FMM3D use almost the same subroutines
for tree construction and the same fast SIMD inverse square-root evaluator
from the \texttt{SCTL} library~\cite{sctl}. Thus, the throughput shown here
are good indicators of the performance of these two algorithms. The top right figure also shows the throughput
for 3D Stokeslet of \texttt{pvfmm}~\cite{pvfmmlib} which implements the 
kernel-independent FMM in \cite{Ying2004} with nearly optimal implementation. The performance of DMK is again very close to that of \texttt{pvfmm}, as for the Laplace kernel reported in~\cite{jiang2025cpam}. 

\begin{figure}[ht]
  \centering
  \includegraphics[height=48mm]{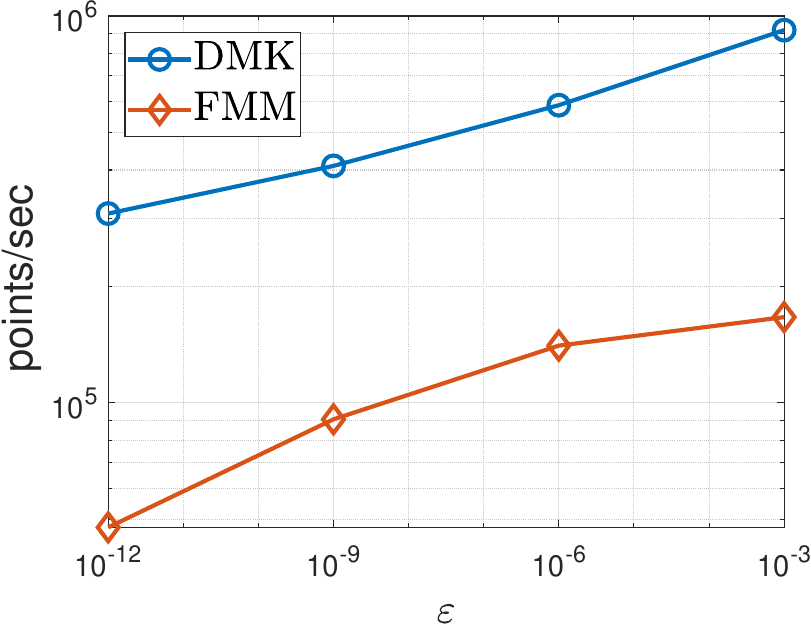}
  \includegraphics[height=48mm]{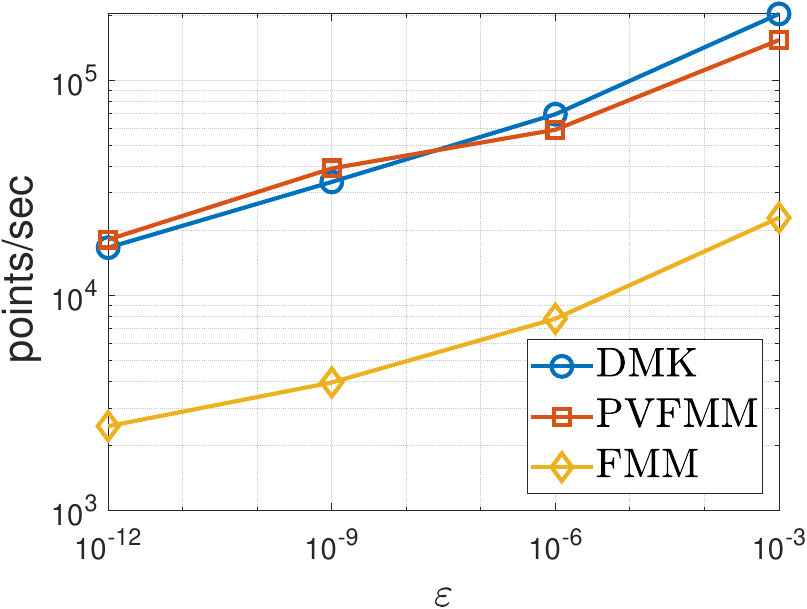}
  
  \vspace{4mm}
  
  \includegraphics[height=48mm]{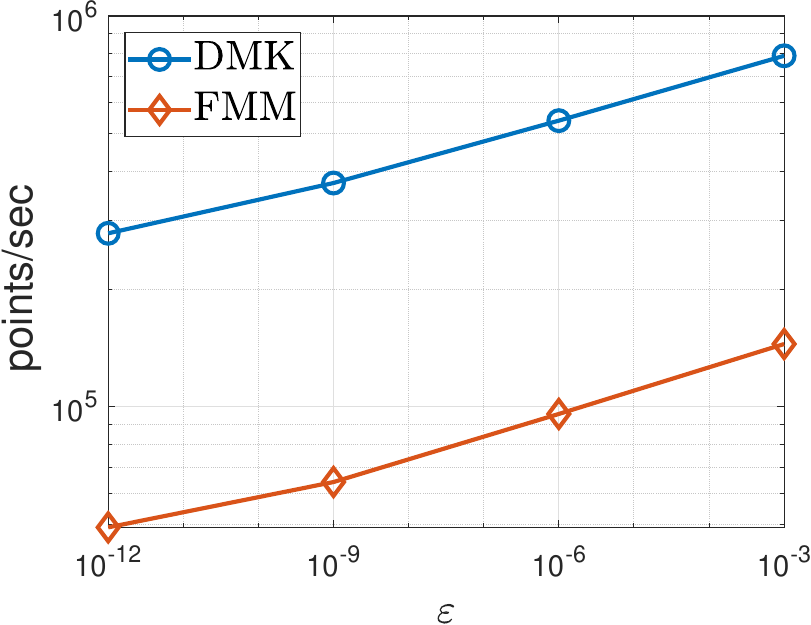}
  \includegraphics[height=48mm]{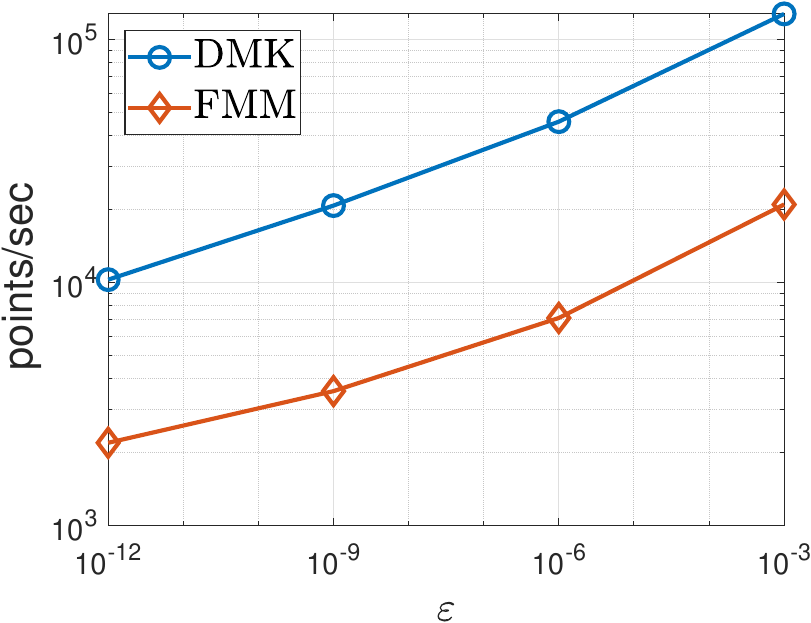}

  \caption{Average throughput of DMK and FMM for the Stokeslet and stresslet in two and
    three dimensions. Top: Stokeslet; bottom: stresslet. Left: 2D; right: 3D.}
  \label{fig:throughput}
\end{figure}

\section{Conclusions} \label{sec:conclusions}
We have created a framework for constructing a splitting of the biharmonic kernel for any even, normalized, window function. With a Gaussian window, the classical decomposition first derived by Hasimoto is obtained. A new and more efficient kernel split is obtained, however, using the prolate spheroidal wave function (PSWF) of order zero instead. More efficient here means that, for the same spatial support of the resulting residual kernel, fewer Fourier modes are needed to resolve the mollified kernel to the desired accuracy. 

The kernel splitting is carried out in Fourier space and the prolate kernel can be used to accelerate fast Ewald summation or the DMK framework for any kernel derived from the biharmonic Green's function in any dimension. This includes Stokeslet, stresslet, and elastic kernels, as well as other scalar kernels encountered in solving boundary value problems for the biharmonic equation. 

In the DMK framework, the two most important parameters affecting the computational cost are the number of Fourier modes ($N_1^d$) and Chebyshev points ($p^d$) needed for a given accuracy; they are the same for each level in the tree. Numerical results show a significant reduction of these numbers when switching from Gaussians to prolates. The performance of the DMK method with prolates compares favorably with several distinct implementations of the fast multipole method (FMM).

We have also described the implementation of periodic boundary conditions on a cube in the DMK framework, with run times comparable to the free space case. 
The extension to an arbitrary rectangular box with any periodicity (e.g., 1-periodic, 2-periodic, and fully periodic in three dimensions) is currently underway. 

Finally, though prolates are advantageous for convolution with point sources, kernel splitting based on Gaussians may be preferable when dealing with continuous sources, because the separability of the Gaussian can be used to accelerate the finest level contributions of the singular residual kernel (see, for example, \cite{greengard2024sirev,jiang2025cpam}).

\section*{Acknowledgements}
Tornberg acknowledges the support from the Swedish Research Council under grant no 2023-04269.

\bibliographystyle{abbrvnat_mod}
\bibliography{dmk_stokes} 

\appendix
\begin{appendices}
\section{Derivations of residual Stokes kernels}
\label{sec:stokes-split-deriv}

Let $D_S$ and $D_T$ denote the differential operators in defining the
Stokeslet~\eqref{eq:stokeslet_def_deriv} and the
stresslet~\eqref{eq:stresslet_def_deriv},
\begin{align}
  D_S &= -\delta_{jl}\nabla^2 + \nabla_j\nabla_l, \label{eq:stoleslet_op}\\
  D_T &=  -\left(
        \delta_{jl}\nabla_m+\delta_{lm}\nabla_j+\delta_{mj}\nabla_l
        \right) \nabla^2 + 2\nabla_j\nabla_l\nabla_m. \label{eq:stresslet_op}
\end{align}
For any radially symmetric function $f(r)$, application of $D_S$ and $D_T$ 
  leads to 
\be
    D_S[f](\vx) 
    =  -\left( (d-2)f'(r) + r f''(r)\right) \frac{\delta_{jl}}{r}
       + \left( - f'(r) + r f''(r) \right)  \frac{x_jx_l}{r^3} .
       \label{eq:dsf}
\ee
and
\be
\ba
  D_T[f](\vx) = &\frac{\delta_{jl}x_m + \delta_{lm}x_j + \delta_{mj}x_l}{r^3}
  \del{ (d-3)\del{f'(r)-r f''(r)} - r^2 f'''(r)} \\
  &- 6\frac{x_jx_lx_m}{r^5} 
   \del{ -f'(r) + r f''(r) - \tfrac{1}{3} r^2  f'''(r)} .
\ea
  \label{eq:dtf}
\ee

Given a general kernel split of the biharmonic in $\R^d$, $d=2,3$, the
residual Stokeslet $S_R$ and the residual stresslet $T_R$ can be
derived by application of $D_S$ and $D_T$ to the residual biharmonic
$B_R(r)$. Rescaling to match the factors of $S$ and $T$ when $d=3$, we define $S_R$ and $T_R$ as

\begin{align}
    S_{R,jl}(\vx)
    &=   \RSdiag(r) \frac{\delta_{jl}}{8\pi r}
      +   \RSoffd(r)  \frac{x_jx_l}{8\pi r^3},\\
  T_{R,jlm}(\vx) &= \frac{\delta_{jl}x_m + \delta_{lm}x_j + \delta_{mj}x_l}{8\pi r^3} \, \RTdiag(r)
  - \frac{3}{4\pi}\frac{x_jx_lx_m}{r^5} \, \RToffd(r),
\end{align}
where
\begin{align}
  \RSdiag(r) &= -8\pi\left( (d-2)\Bres'(r) + r\Bres''(r) \right), \\
  \RSoffd(r) &= 8\pi\left( -\Bres'(r) + r\Bres''(r) \right), \\
  \RTdiag(r) &= 8\pi\del{ (d-3)\del{\Bres'(r)-r\Bres''(r)} -r^2\Bres'''(r)}, \\
  \RToffd(r) &= 8\pi \del{ -\Bres'(r) + r\Bres''(r) - \tfrac{1}{3} r^2 \Bres'''(r)} .
\end{align}
Similarly, the residual rotlet is derived from
  \eqref{eq:rotlet_def_deriv} by differentiating a split of the
  harmonic kernel,
\begin{align}
    \Omega_{R,jl} = \frac{1}{8\pi}\epsilon_{jlm}\frac{x_m}{r^3}
    \underbrace{\del{-4 \pi r^2 H_R'(r)}}_{\ROoffd(r)}
\end{align}
Using the residual biharmonic kernel \eqref{eq:BR_gen} and the residual harmonic kernel \eqref{eq:harmonic_split}, the above expressions generate the residual kernels of \cref{sec:stokes-split}.

\subsection{Self-interaction}
\label{sec:self-inter-deriv}

We here derive the self-interaction terms of the 3D Stokes kernels of \cref{sec:stokes-split}. For the Stokeslet,
\be
\ba
    S_{M,jl}(0) &= \lim_{r\to 0} \sbr{
       \del{1-\RSdiag(r)} \frac{\delta_{jl}}{8\pi r}
     + \del{1-\RSoffd(r)} \frac{\hat x_j \hat x_l}{8\pi r}
     },
     \\
     &=
     \lim_{r\to 0} \left[
       \del{\frac{1}{r}\int_0^r \wfunc(s) \dif s + \wfunc(r)} \frac{\delta_{jl}}{4\pi}
       +
       \del{\frac{1}{r}\int_0^r \wfunc(s) \dif s - \wfunc(r)} \frac{\hat x_j \hat x_l}{4\pi}
       \right]
  \\
  &= \frac{1}{2\pi} \wfunc(0) \delta_{jl},
\ea
\ee
where we have used l'H\^{o}pital's rule and the convenience notation $\hat x_j=x_j/r$. From this, we obtain \eqref{eq:S_self}.

For the stresslet, we need to determine  
\begin{align}
    T_{M,jlm}(0)
    = \lim_{r \to0} \Bigg[\frac{1}{8\pi}
        \del{\delta_{jl}\hat x_m + \delta_{lm}\hat x_j + \delta_{mj}\hat x_l} \frac{-\RTdiag(r)}{r^2}
        - \frac{3}{4\pi}\hat x_j \hat x_l \hat x_m
          \frac{1-\RToffd(r)}{r^2}
             \Bigg]
             .
\end{align}
Evaluating the limits, again using l'H\^{o}pital's rule, we get
\begin{align}
  \lim_{r \to 0} \frac{ -\RTdiag(r) }{r^2}
  &= \lim_{r \to 0} 2 \wfunc'(r) = 0,  \\
  \begin{split}
    \lim_{r\to 0} \frac{1-\RToffd(r)}{r^2}
    &=\lim_{r \to 0} \del{
      \frac{ 2\int_0^r \wfunc(t) \dif t - 2 r \wfunc(r)}{r^2}
      + \tfrac{2}{3} \wfunc'(r)
      } \\
    &= -\wfunc'(0) + \tfrac{2}{3} \wfunc'(0)
      = 0.
  \end{split}
  \label{eq:MToffd_limit}
\end{align}
From this we find that the stresslet self-interaction is zero.

The self-interaction term for the rotlet is also zero, as 
\begin{align}
  \Omega_{M,jl}(0) =
  \lim_{r\to 0} \del{\epsilon_{jlm}\hat x_m \frac{1-\ROoffd(r)}{8\pi r^2}  } = 0 .
\end{align}
This follows from a limit similar to that in \eqref{eq:MToffd_limit},
\begin{align}
  \lim_{r\to 0}  \frac{1-\ROoffd(r)}{r^2} 
  = \lim_{r\to 0} \frac{ 2\int_0^r \wfunc(t) \dif t - 2 r \wfunc(r)}{r^2}
  = -\wfunc'(0) = 0.
\end{align}

  \section{Numerical evaluation of Stokeslet and stresslet residual kernels}
  \label{app:num_fourier_split}

  For a given biharmonic split, with the mollified kernel defined in
  Fourier space as $\Bmollhat(k) = \widehat B(k) \widehat \gamma(k)$,
  the residual kernel $\Bres(r)$ may not be explicitly available. It
  can then be computed, and subsequently differentiated, numerically
  without loss of accuracy or efficiency. The procedure for this,
  outlined below, was first introduced in \cite{jiang2025cpam}.

  The task is to compute the smooth mollified kernel $S_{M,jl}$ for
  the Stokeslet and $T_{M,jlm}$ for the stresslet, from which the
  residual kernels follow as $S_{R,jl}(r) = S_{jl}(r) - S_{M,jl}(r)$
  and $T_{R,jlm}(r) = T_{jlm}(r) - T_{M,jlm}(r)$.  It is clear that
  the mollified kernels of Stokeslet and stresslet can be obtained by
  applying the differential operators $D_S$ \eqref{eq:stoleslet_op}
  and $D_T$ \eqref{eq:stresslet_op} to the biharmonic mollified kernel
  $\Bmoll(r)$.  Although we know the Fourier transform $\Bmollhat$,
  the singularity in $\widehat B(k)$ at $k=0$ prevents us from
  computing $\Bmoll$ directly through an inverse Fourier
  transform. Instead, we first replace $\widehat B$ with the windowed
  biharmonic \eqref{eq:Btrunc_hat} truncated outside a window radius
  $\truncrad$, and work with the truncated mollified biharmonic
  $\Bmollhat^\truncrad$, defined in \eqref{eq:Bmollhat_trunc}.
  
Recall now that the radial Fourier transform kernel is given by 
\begin{align}
    F_k(r) = \begin{cases}
        J_0(kr), & d=2, \\
        \frac{\sin(kr)}{kr}, &d=3
    \end{cases}
\end{align}
and the radial inverse Fourier transform is
given by 
the formula 
\be
f(r) = \frac{1}{2\pi^{d-1}}\int_0^\infty F_k(r)\widehat{f}(k)k^{d-1}dk.
\ee
Thus,
\be
\Bmoll^\truncrad(r) = \frac{1}{2\pi^{d-1}}\int_0^\infty F_k(r)\Bmollhat^\truncrad(k)k^{d-1}dk
=\frac{1}{2\pi^{d-1}}\int_0^\infty F_k(r)\widehat{B}^\truncrad(k)\wfunhat(k)k^{d-1}dk
\ee
It is easy to see that the differentiation operators in \eqref{eq:stokeslet_def_deriv}
and \eqref{eq:stresslet_def_deriv}
can be passed to the radial Fourier transform kernel $F_k(r)$, and we will use this fact to compute the mollified kernel of the Stokeslet and stresslet, that we write on the forms: 
\be
S_{M,jl}(\vx) 
=  \MSdiag(r)  \frac{\delta_{jl}}{8\pi r}
+  \MSoffd(r)  \frac{x_jx_l}{8\pi r^3} ,
\label{eq:win_stokeslet_detail}
\ee
\be
T_{M,jlm}(\vx) = \frac{\delta_{jl}x_m + \delta_{lm}x_j + \delta_{mj}x_l}{8\pi r^3}
\MTdiag(r)
- \frac{3}{4\pi} \frac{x_jx_lx_m}{r^5} 
\MToffd(r) ,
\label{eq:win_stresslet_detail}
\ee
The Stokeslet and stresslet are obtained by differentiating $B(r)$ in \eqref{eq:biharmonic_green_def}, rather than $B^\truncrad(r)$ in \eqref{eq:Btrunc}; the same applies to the associated mollified kernel.
We have
\be
\ba
\MSdiag(r) &=
\MSdiag^\truncrad(r) - 8\pi\left[
(d-2)(B(r)-B^\truncrad(r))' + r (B(r)-B^\truncrad(r))''\right]\\
&=
\begin{cases}
\MSdiag^\truncrad(r)
+2 r \left(1-\log \truncrad\right)   , & d=2, \\      \MSdiag^\truncrad(r)+\frac{2 r}{\truncrad} , &d=3. 
\end{cases}
\ea
\label{eq:msdiag}
\ee
For the other three kernels, the action of the associated differential operators on 
$B(r)-B^\truncrad(r)$
is identically zero. 
Hence, 
$\MSoffd(r)=\MSoffd^\truncrad(r)$, 
$\MTdiag(r)=\MTdiag^\truncrad(r)$, 
and $\MToffd(r)=\MToffd^\truncrad(r)$, where
\be
\ba
\MSdiag^\truncrad(r) &= \frac{1}{2\pi^{d-1}}\int_0^\infty \FSdiag(k,r)\widehat{B}^\truncrad(k)\wfunhat(k)k^{d-1}dk, \\
\MSoffd^\truncrad(r) &= \frac{1}{2\pi^{d-1}}\int_0^\infty \FSoffd(k,r)\widehat{B}^\truncrad(k)\wfunhat(k)k^{d-1}dk,\\
\MTdiag^\truncrad(r) &= \frac{1}{2\pi^{d-1}}\int_0^\infty \FTdiag(k,r)\widehat{B}^\truncrad(k)\wfunhat(k)k^{d-1}dk,\\
\MToffd^\truncrad(r) &= \frac{1}{2\pi^{d-1}}\int_0^\infty \FToffd(k,r)\widehat{B}^\truncrad(k)\wfunhat(k)k^{d-1}dk,
\ea
\label{eq:STterms}
\ee
with
\begin{align}
  \begin{split}
    \FSdiag(r) &= -8\pi\left( (d-2)F_k'(r) + rF_k''(r) \right), \\
    \FSoffd(r) &= 8\pi\left( -F_k'(r) + rF_k''(r)  \right), \\
    \FTdiag(r) &= 8\pi\del{ (d-3)\del{F_k'(r)-rF_k''(r)} -r^2F_k'''(r)}, \\
    \FToffd(r) &= 8\pi \del{ -F_k'(r) + rF_k''(r) - \tfrac{1}{3} r^2 F_k'''(r)} .
  \end{split}
  \label{eq:dFterms}
\end{align}
The integrands in \eqref{eq:STterms} are all smooth because $F_k(r)$ is a smooth 
even function and the truncated biharmonic kernel $\widehat{B}^\truncrad(k)$ is smooth. Moreover, the biharmonic mollifier $\wfunhat(k)$ goes to zero rapidly. The integrals can therefore be truncated and accurately evaluated using the Gauss-Legendre quadrature for
any given values of $r$ and $\truncrad$, with the latter chosen
according to our range of interest. 

\begin{remark}
The functions $\MSdiag(r)$,
$\MSoffd(r)$, $\MTdiag(r)$, and $\MToffd(r)$ are smooth and can be
accurately represented using polynomial approximations created in
a precomputation step. This means that there is negligible
additional cost associated with having a numerically computed
residual kernel.
\end{remark}

\subsection{Self-interaction}

We now calculate self-interaction terms in the same way as in \cref{sec:self-inter-deriv}.
Both the differentiation operators in \eqref{eq:stokeslet_def_deriv}
and \eqref{eq:stresslet_def_deriv}
and the limit operation
can be passed to the radial Fourier transform kernel $F_k(r)$. For two and three dimensions, straightforward
calculation using \eqref{eq:dFterms} leads to
\be
\lim_{r\to 0} \frac{\FSdiag(r)}{r}
=8\pi\frac{d-1}{d}k^2,\quad
\lim_{r\to 0} \frac{\FSoffd(r)}{r}
 = 0,\quad
\lim_{r\to 0} \frac{\FTdiag(r)}{r^2}
= 0,\quad
\lim_{r\to 0} \frac{\FToffd(r)}{r^2}
 = 0.
\ee
Thus, the self-interaction terms connected to
$\MSoffd$, $\MTdiag$, $\MToffd$
are all equal to zero, with the only nonzero
self-interaction coming from
\be
\lim_{r\to 0}\frac{\MSdiag^\truncrad(r)}{r}=\frac{1}{2\pi^{d-1}}\int_0^\infty \left( \lim_{r\to 0} \frac{\FSdiag(r)}{r}\right) \Bmollhat^\truncrad(k)k^{d-1}dk
=\frac{4(d-1)}{d\pi^{d-2}}\int_0^\infty \Bmollhat^\truncrad(k)k^{d+1}dk,
\label{eq:msdiag0}
\ee
which can be calculated numerically to high accuracy via the Gauss-Legendre quadrature.
Finally, the Stokeslet self-interaction is obtained by 
combining \eqref{eq:msdiag} and \eqref{eq:msdiag0}, and writing
\begin{align}
  S_{M,jl}(0) = \frac{\delta_{jl}}{8\pi} \lim_{r\to 0} \frac{\MSdiag(r)}{r}
\end{align}

Note that we can use $\Bmollhat(k)$
instead of $\Bmollhat^\truncrad(k)$ in \eqref{eq:msdiag0} to obtain $\MSdiag(0)$ directly.  
The resulting expression is
\be
\lim_{r\to 0} \frac{\MSdiag(r)}{r}
=\frac{4(d-1)}{d\pi^{d-2}}\int_0^\infty \Bmollhat(k)k^{d+1}dk
=
\frac{4(d-1)}{d\pi^{d-2}}\int_0^\infty \wfunhat(k)k^{d-3}dk.
\ee
When $d=3$, combining the above equation
with \eqref{eq:gammaB_in_phi}, we 
obtain an explicit expression
for $\MSdiag(0)$:
\be
\ba
\lim_{r \to 0} \frac{\MSdiag(r)}{r} &= 
\frac{8}{3\pi}\int_0^\infty \wfunhat(k)dk
=\frac{4}{3\pi}\int_{-\infty}^\infty \wfunhat(k)dk 
= \frac{8}{3} \wfun(0).
\ea \ee
With our biharmonic mollifier \eqref{eq:biharmonic_screen_R},
this recovers the self-interaction \eqref{eq:S_self}.

For $d=2$, we will need to deal with singular integrals due to the singularity 
at the origin, and it is better to use
\eqref{eq:msdiag0}.

\end{appendices}

\end{document}